\documentclass{amsart}
\usepackage[mathscr]{eucal}
\usepackage{todonotes}
\usepackage{tikz}
\usepackage{xifthen} 
\usetikzlibrary{positioning,calc}
\usepackage[utf8]{inputenc}
\usepackage{amssymb,enumitem}
\usepackage{doi}
\hypersetup{hidelinks}
\usepackage{hyperref}
\usepackage[english]{babel}
\numberwithin{equation}{section}

\usepackage[final]{changes}

\setcommentmarkup{}

\theoremstyle{plain}
\newtheorem{thm}{Theorem}[section]
\newtheorem{prop}[thm]{Proposition}

\newtheorem{cor}[thm]{Corollary}
\newtheorem{claim}{Claim}[section]

\newtheorem{lem}[thm]{Lemma}
\newtheorem{conj}[thm]{Conjecture}
\theoremstyle{definition}
\newtheorem{defn}[thm]{Definition}

\newtheorem{exa}[thm]{Example}

\newtheorem{rem}[thm]{Remark}

\newenvironment{pf}{\begin{proof}}{\end{proof}}

\newcommand{\R}{\mathbb R}
\newcommand{\Q}{\mathbb{Q}}
\newcommand{\N}{\mathbb N}
\newcommand{\Z}{\mathbb{Z}}

\newcommand{\Rexp}{\mathbb{R}_{\exp}}
\newcommand{\Rrexp}{\mathbb{R}_{\texp}}
\newcommand{\Twsc}{T_{\textnormal{wsc}}}

\DeclareMathOperator{\img}{Im}

\DeclareMathOperator{\spn}{span}
\DeclareMathOperator{\td}{td}
\DeclareMathOperator{\etd}{etd}
\DeclareMathOperator{\ecl}{ecl}

\newcommand{\Jac}{\textnormal{Jac}}

\DeclareMathOperator{\ldim}{ldim}

\DeclareMathOperator{\germ}{germ}

\newcommand{\liff}{\leftrightarrow}
\newcommand{\nin}{\not\in}

\DeclareMathOperator{\reg}{reg}

\DeclareMathOperator{\texp}{e} 

\DeclareMathOperator{\Exp}{Exp}
\newcommand{\E}{\texp} 

\newcommand{\eps}{\varepsilon}

\newcommand{\SC}{\text{SC}_{\R}}

\newcommand{\Tres}{T^{dc}_{\texp}}
\newcommand{\Tglob}{T^{dc}_{\exp}}

\DeclareMathOperator{\Fin}{Fin}

\newcommand{\M}{M}

\newcommand{\res}{\operatorname{res}}
\newcommand{\rs}[1]{\underline{#1}}
\newcommand{\st}{\operatorname{st}} 
\newcommand{\rf}{A/\mathfrak{m}_A}

\DeclareMathOperator{\dom}{dom}

\newcounter{paranum}[section]
\renewcommand{\theparanum}{\thesection.\arabic{paranum}}

\newcommand{\numpar}[1][]{%
    \refstepcounter{paranum}
\textbf\textnormal{(\theparanum)}
    \ifthenelse{\isempty{#1}}
        {\hspace{0.5em}}
        {\ \textnormal{#1}.}
}

\title{On the elementary theory of the real exponential field}

\author{Alessandro Berarducci}
\address[A. Berarducci]{Dipartimento di Matematica, Università di Pisa, Largo Bruno Pontecorvo 5, 56127 Pisa, Italy}
\email{alessandro.berarducci@unipi.it}
\author{Francesco Gallinaro}
\address[F. Gallinaro]{Centro di Ricerca Matematica Ennio De Giorgi, Scuola Normale Superiore, Piazza dei Cavalieri 3, 56126 Pisa, Italy}
\email{francesco.gallinaro@sns.it}

\date{\replaced{24 June 2026}{9 March 2026}}
\thanks{Partially supported by the Italian research project PRIN 2022, Models, sets and classifications, Prot.\,2022TECZJA.\\ \emph{Statements and declarations}: The authors have no competing interests to declare.}

\subjclass[2020]{03C10, 03C64, 12L12}
\keywords{Exponentiation, model completeness, definable completeness, o-minimality}

\begin{document}

\begin{abstract}
    Assuming Schanuel's conjecture, we prove that the complete theory $T_{\exp}$ of the real exponential field is axiomatized by the axioms of definably complete exponential fields satisfying $\exp' = \exp$. This implies the result of Macintyre and Wilkie that, under the same conjecture, $T_{\exp}$ is decidable. Our approach is based on the model completeness of a similar set of axioms for the exponential function restricted to $(-1,1)$, which we prove unconditionally. 
\end{abstract}

	\maketitle

	\begin{minipage}{0.9\textwidth}
		\tableofcontents
	\end{minipage}
	\vskip1\baselineskip

\section{Introduction}
In \cite{Tarski1951}, Tarski proved that the (first-order) theory of the field of real numbers $\R$ admits quantifier elimination and is decidable, that is, there exists an algorithm to determine whether a first-order sentence $\varphi$ in the language $L=\{\leq, 0, 1, +, \cdot\}$ of ordered rings is true in $\mathbb{R}$. A modern treatment can be found in \cite[Theorem~6.41]{Poizat2000}, including a proof that the theory of $\R$ is axiomatized by the axioms of real closed ordered fields. 

In the same paper, Tarski asked whether the decidability result could be extended if one adds exponentiation to the language, considering the real exponential field $\R_{\exp} = (\R, \exp)$. A first obstacle is that the complete theory $T_{\exp}$ of $\mathbb{R}_{\exp}$ does not admit quantifier elimination due to a counterexample of Osgood \cite{Osgood1915} (see also \cite[Equation (1.1)]{Gabrielov2007}). Nevertheless, in \cite{Wilkie1996}, Wilkie proved that $T_{\exp}$ is {\em model complete}: every formula is equivalent to an existential formula (a quantifier-free formula preceded by a block of existential quantifiers). 
In geometric terms, this means that every definable set in the structure $\R_{\exp}$ is a projection of a set defined by equations of the form $p(x_1, \ldots, x_n, e^{x_1}, \ldots, e^{x_n}) = 0$, where $p$ is a polynomial. The sets defined by such equations have a finite number of connected components by \cite{Khovanskii1981} (see also \cite[Theorem 5.3]{Wilkie1996}), hence the same property holds for every definable set in $\R_{\exp}$. In particular, $\R_{\exp}$ is {\em o-minimal}: every subset of $\R$ definable in $\R_{\exp}$ is a finite union of intervals. Definable sets in o-minimal structures share many of the finiteness properties of semialgebraic sets, see \cite{vdDries1998} for a systematic treatment. Wilkie's results on $\mathbb{R}_{\exp}$ have been a driving motivation for the study of o-minimality. 

In \cite{Macintyre1996a} Macintyre and Wilkie proved the decidability of the theory of $\R_{\exp}$ assuming the real form of {\em Schanuel's conjecture} ($\SC$ for short): given a $n$-tuple of $\Q$-linearly independent real numbers and their exponentials, the resulting $2n$-tuple has transcendence degree at least $n$ over $\Q$. More precisely, Macintyre and Wilkie show that, if $\SC$ holds, then there is an algorithm to test whether an exponential polynomial $p(x_1, \ldots, x_n, e^{x_1}, \ldots, e^{x_n})$ over $\Q$ has a real zero. They then prove that the existence of such an an algorithm implies the decidability of $T_{\exp}$. 

Here we prove, assuming $\SC$, that $T_{\exp}$ \deleted{can be} \added{is} axiomatized by the \added{theory} \deleted{axioms} $\Tglob$ of \emph{definably complete} exponential fields satisfying $\exp' = \exp$ (Theorem \ref{thm:complete-unrestricted}). This gives a positive answer, modulo $\SC$, to a conjecture considered in \cite[Section 5]{Miller1996}, \cite[Theorem 2.8]{Berarducci2004c}, \cite[Theorem 4.7.25]{ServiTesi2008}, \cite{Krapp2023}. 

We recall that an ordered structure is definably complete if every bounded definable subset of its domain has a least upper bound (a first order version of Dedekind completeness). An ordered field is real closed if and only if it is definably complete, hence $\Tglob$ is a natural extension of the theory of real closed fields. 
Since a complete theory is decidable if and only if it is recursively axiomatized, we also obtain a new proof of the decidability result of \cite{Macintyre1996a}. 

Towards the proof that $\SC$ implies the completeness of $\Tglob$, in Theorem \ref{thm:model-complete} we establish unconditionally the model completeness of a set of axioms $\Tres$ (very similar to $\Tglob$) \added{satisfied by the real field equipped with the exponential function restricted to $(-1,1)$.}\deleted{for the exponential function restricted to $(-1,1)$.} The models of $\Tres$ are the definably complete structures \added{$M$} expanding an ordered field \deleted{equipped with} \added{with a function $\texp:M\to M$ which satisfies $\texp(0)=1, \texp'(x) = \texp(x)$ for $x\in (-1,1)$ and $\texp(x)=0$ for $x\nin (-1,1)$.} \deleted{an exponential function $\texp$ restricted to $(-1,1)$ and satisfying $\texp' = \texp$ on that interval.} 

Granted the model completeness of $\Tres$, assuming $\SC$ we deduce the completeness of $\Tres$ (Theorem \ref{thm:complete-restricted}) by showing that the prime submodel of $\R_{\texp}$ embeds in every models of $\Tres$ (in Remark \ref{rem:Krapp} we relate this to a similar embedding result of Krapp). 
The completeness of $\Tglob$ (Theorem \ref{thm:complete-unrestricted}) in turn follows by a result of Ressayre \cite{Ressayre1993}.  

In the last section we strengthen the embedding result above by showing \added{that if $N$ is a model of $\Tres$ and $A$ is a convex subring of $N$ which contains an elementary substructure of $N$, then the residue field of $A$, with its induced restricted exponential function, embeds into $N$}\deleted{, under $\SC$ (in fact, under a weaker hypothesis), that the residue field of a convex subring of a model $M$ of $\Tres$ can be embedded in $M$ preserving the restricted exponential function} (Theorem \ref{thm:embedding}). 

To prove the model completeness of $\Tres$ we work by induction on the complexity of formulas as in Wilkie's proof of model completeness of expansions of $\mathbb{R}$ by restricted Pfaffian functions \cite{Wilkie1996}. 
As in that paper, a crucial step consists in bounding the norm of regular solutions, here called Khovanskii points, to certain $n \times n$ systems of equations. Our context is different because we cannot assume that $\Tres$ is a complete theory, but we can take advantage of the fact that models of $\Tres$ are o-minimal by the results in \cite{FornasieroServi2010,Hieronymi2011}. 

A major difficulty, however, is that we are not able to prove that $\Tres$ is a polynomially bounded theory (i.e., in every model, every definable function in one variable is ultimately bounded by a polynomial), so we have to depart from the approach in \cite{Wilkie1996} which uses in an essential way 
\added{the fact that for every univariate function $g$ definable in an expansion of $\mathbb{R}$ by restricted analytic functions (see \cite{VandenDries1986,Denef1988}), if $g$ is not eventually zero, then there is a rational number $q$ and a non-zero constant $c$ such that $g(x)/cx^q$ tends to $1$ for $x\to +\infty$ \cite[Corollary 3.5]{Wilkie1996}. } \deleted{consequences of the polynomial boundedness of the theory $T_{\text{an}}$ studied in  \cite{VandenDries1986,Denef1988}, see in particular }

In our approach, given a model $M$ of $\Tres$, we consider the residue field $R = \mathcal O/\mathfrak m$ of the local ring $\mathcal O$ of germs at $+\infty$ of polynomially bounded definable functions $f:M\to M$. We show that $R$ has an induced restricted exponential function and a compatible derivation. We then apply Ax's functional transcendence theorem in \cite{Ax1971} to prove that, if for a contradiction there is an unbounded Khovanskii point, then two specific germs of polynomially bounded definable functions are not equal modulo the maximal ideal $\mathfrak m$. Using the inductive hypothesis, this allows us to reach another contradiction and establish the boundedness of Khovanskii points. 

 We have seen that, under $\SC$, the theory $\Tres$ coincides with the complete theory of $\R_{\texp}$, so it is polynomially bounded. The role of $\SC$ is precisely to ensure that any model of $\Tres$ has an archimedean submodel. An unconditional proof of polynomial boundedness remains an open problem even after the model completeness has been established.     

 \added{In \cite[\S 5]{Macintyre1996a} the authors show that the decidability of $T_{\exp}$ is equivalent to WSC, a weak form of Schanuel's conjecture which can be expressed by a first-order axiom scheme $\Twsc$ (Definition \ref{def:Twsc}). In the same spirit we note in Remark \ref{rem:Krapp}(\ref{rem:WSC}) that under WSC the theories $\Tres+ \Twsc$ and $\Tglob+\Twsc$ are complete.} 

A large body of papers deal with tameness properties of semianalytic or subanalytic sets \cite{Lojasiewicz1971,VandenDries1986,Denef1988,Gabrielov1968} that can be used to study o-minimal expansions of the reals. These approaches require tools, such as the Weierstrass preparation theorem, that are not available in the non-archimedean context. Other papers, such as \cite{Jones2008,LeGal2008}, assume forms of local polynomial boundedness as a hypothesis. Possible approaches to the polynomial boundedness of $\Tres$ may come from Miller's dychotomy results \cite{Miller1994,Miller1996} or attempts to extend the results of \cite{Rolin2015} to the non-archimedean context, but this seems rather challenging. 

\subsection{Organization of the paper} In Sections \ref{sec:preliminaries} and \ref{sec:ordered} we give some model-theoretic preliminaries. In Sections \ref{sec:exp-fields} and \ref{sec:conv-val-rings} we give the definitions of exponential field, restricted exponential field, and partial exponential field.
 We show how to extend a restricted exponential function to a partial exponential function and how to induce restricted and partial exponential functions on the residue field of a convex subring of a model of $\Tres$. 
 
 In Section \ref{sec:tan-spaces} we review some basic results about differentiable definable functions and definable manifolds in the o-minimal setting. This will be useful  in connection with the definition of Khovanskii point in Section \ref{sec:res-exp-pol}. In Section \ref{sec:exp-alg} we use results of Kirby to relate Khovanskii points to a notion of exponential algebraicity.
 
In Section \ref{sec:germs} we consider the local ring $\mathcal O$ of germs at $+\infty$ of polynomially bounded definable functions in a model of $\Tres$, and equip its residue field $R$ with a restricted exponential function and a compatible derivation. In Section \ref{sec:Ax} we give a uniform version of Ax's theorem for real closed differential fields. 
In Sections \ref{sec:bounded} and \ref{sec:model-complete} we prove the model completeness of $\Tres$.

In Section \ref{sec:lifting} we give an argument to lift Khovanskii points over the residue field of a convex subring of a model of $\Tres$. This argument is used in Section \ref{sec:completeness} to prove the completeness of $\Tres$ assuming $\SC$. In Section \ref{sec:embedding} we prove the embedding theorem for the residue restricted exponential field mentioned above.

\subsection*{Acknowledgements} A version of Theorem \ref{thm:embedding} and various preliminary results thereof, appears in unpublished work by the first author and Marcello Mamino. We thank Mamino for allowing us to include the result in this paper.  A special case of Theorem \ref{thm:embedding} is used to deduce the completeness of $\Tres$ from its model completeness, assuming $\SC$. 

The first author thanks Antongiulio Fornasiero for numerous discussions on the topics of this paper and for sharing some insights that could be used in future attempts to prove the polynomial boundedness of $\Tres$. 

\added{We thank Kobi Peterzil for a very careful reading of a preliminary version of this paper and many helpful comments.}  

\section{Model-theoretic preliminaries}\label{sec:preliminaries}
We assume familiarity with the model-theoretic notions of a {\em structure} and a (first-order) {\em formula} in a given {\em language}~$L$. A {\em sentence} is a formula without free variables. Given a language $L$, we write $M\models \varphi$ to express the fact that $\varphi$ is an $L$-sentence, $M$ is an $L$-structure, and $\varphi$ holds in $M$.  A comprehensive introduction to model theory can be found in~\cite{Hodges}.

A {\em theory} is a collection of sentences in a given language. Given a theory $T$ and a structure $M$ in the same language, we say that $M$ is a {\em model} of $T$ if $M \models \varphi$ for every $\varphi \in T$. We write in this case $M\models T$. 

We write $T\vdash \varphi$ if the formula $\varphi$ is {\em provable} in the theory $T$, or equivalently (by Gödel's completeness theorem), $\varphi$ holds in every model of $T$. 
A theory $T$ is {\em complete} if, for every sentence $\varphi$ in its language, either $\varphi$ or its negation $\lnot \varphi$ is provable in $T$, but not both. 

Given a structure $M$ in a language $L$, the theory $\text{Th}(M)$ of $M$ is the collection of all sentences of $L$ that are true in $M$. Clearly $\text{Th}(M)$ is a complete theory. 

Two structures $M$ and $N$ are {\em elementarily equivalent}, written $M\equiv N$, if they have the same complete theory. The theory $\textsf{RCF}$ of real closed fields (in the language $L = \{\leq,0,1,+,\cdot\}$) is complete, so its models are elementarily equivalent to the ordered field $\R$ of real numbers \cite{Tarski1951}.  

If $\varphi$ is a formula with free variables included in $\{x_1, \ldots, x_n\}$, after fixing an ordering of the variables, we may write $\varphi$ in the form $\varphi(x_1, \ldots, x_n)$. In this case, if $(a_1, \ldots, a_n)\in M^n$, we write $M\models \varphi(a_1, \ldots, a_n)$ if $\varphi$ holds in $M$ when we assign to the variables $x_1, \ldots, x_n$ the values $a_1, \ldots, a_n$ respectively. A set $X \subseteq M^n$ is {\em definable} if there are a formula $\varphi(x_1, \ldots, x_n, y_1, \ldots, y_k)$ and parameters $b_1, \ldots, b_k$ such that $X = \{(a_1, \ldots, a_n)\in M^n : M \models \varphi(a_1, \ldots, a_n, b_1, \ldots, b_k)\}$. 
 We say in this case that $X$ is defined by the {\em formula with parameters}  $\phi(x_1,\ldots, x_n) = \varphi(x_1, \ldots, x_n, b_1, \ldots, b_k)$ and we write $X = \phi(M)$. For instance a circle of radius $r$ is defined in $\R$ by the formula $x_1^2+x_2^2=r^2$, with parameter $r$. 

An {\em expansion} of a structure $\mathscr M$ is a structure $\mathscr{M}'$ with the same domain and a possibly larger language $L'\supseteq L$ in which the symbols of $L$ are interpreted as in $\mathscr M$. For example an ordered ring $\mathscr{M}' = (M, \leq, 0,1,+,\cdot)$ is an expansion of its underlying ordered additive group $\mathscr{M} = (M, \leq, 0,+)$. If $\mathscr M'$ is an expansion of $\mathscr M$, we say that $\mathscr M$ is a {\em reduct} of $\mathscr M'$. When the language is understood from the context, we use the same notation for a structure and its underlying domain. Thus, in the case of an ordered ring, we may write $M$ instead of $(M,\leq,0,1,+,\cdot)$. 

Expansions should not be confused with extensions. We say that a structure $M$ is an {\em extension} of a structure $M$, written $M \supseteq N$, if $M$ and $N$ are structures in the same language and $M$ is a {\em substructure} of $N$.

If $N \supseteq M$ is an extension of $M$ and $X = \phi(M) \subseteq M^n$ is a definable set in $M$, we write $\phi(N)$ for the set defined in $N$ by the same formula $\phi$ (where $\phi$ may contain parameters from $M$). If $\phi$ is quantifier-free, then $\phi(M) = \phi(N) \cap M^n$, but in general this equality fails. For instance if $\phi(x)$ is the formula $\exists y (x = y^2)$, then $\phi(\Q) \neq \phi(\R) \cap \Q$. 

We say that $M$ is an {\em elementary substructure} of $N$, written $M\preceq N$, if $\phi(M) = \phi(N) \cap M^n$ for every formula $\phi = \varphi(x_1, \ldots, x_n)$ with parameters from $M$. Note that $M \preceq N$ implies $M\equiv N$.   

A theory $T$ is {\em model complete} if and only if, for any pair of models $M \subseteq N$ of $T$, we have $M\preceq N$. Model completeness can also be characterized by the property that every formula $\varphi(\bar x)$ in the language of $T$, where $\bar x$ is a finite tuple of variables, is equivalent, in all models of $T$, to an existential formula $\exists \bar y \theta(\bar x, \bar y)$ (with $\theta$ quantifier free).
This is also equivalent to the fact that, for any pair of models $M \subseteq N$ of $T$, 
$M$ is {\em existentially closed} in $N$, namely for every existential formula $\exists \bar x \theta(\bar x, \bar a)$ with parameters $\bar a$ from $M$, we have $M \models \exists \bar x \theta(\bar x, \bar a) \iff N \models \exists \bar x \theta(\bar x, \bar a)$, where $\bar x$ is a finite tuple of variables. A proof of these equivalences can be found in \cite[Theorem 1.2]{wilkieLecturesEliminationTheory2015} or in \cite[Theorem 8.3.1]{Hodges}. 

The theory $\textsf{RCF}$ of real closed fields is model complete, since it has the stronger property of admitting {\em quantifier elimination} \cite{Tarski1951}, see \cite[Theorem 6.41]{Poizat2000} for a modern treatment.

\section{Definable completeness and o-minimality} \label{sec:ordered}

If $M = (M,\leq, \ldots)$ is an expansion of an ordered structure, we put on $M$ the topology generated by the open intervals, and on $M^n$ the product topology. If $M$ is an expansion of a non-archimedean ordered field, this topology is neither locally compact nor locally connected. This is the typical situation we shall encounter in this paper. 

Given an $L$-structure $\mathscr K = (K, \leq, \ldots)$ which expands a dense totally ordered set $(K,\leq)$ without end-points,  we say that $\mathscr K$ is {\em definably complete} if every non-empty definable subset of $K$ that is bounded from above has a least upper bound in $K$. We recall that $\mathscr K$ is {\em o-minimal} if every definable subset of $K$ is already definable in the reduct $(K, \leq)$. Equivalently, $\mathscr K$ is o-minimal if every definable subset of $K$ is a finite union of intervals. 
We assume some familiarity with the basic results on o-minimal structures, including the definition of {\em dimension} of a definable set and the {\em cell decomposition} theorem. A standard reference is \cite{vdDries1998}. Given an o-minimal structure $M$, or more generally a definably complete structure, a definable subset of $M^n$ is called {\em definably compact} if it is closed and bounded, and {\em definably connected} if it is not the union of two non-empty definable clopen subsets. 
Unlike o-minimality, the notion of definable completeness is given by a set of (first-order) axioms, in fact an axiom scheme.

In general o-minimality implies definable completeness, but the converse is not always true. For instance $(\R, \leq, 0,1,+,\cdot, \sin)$ is not o-minimal (because the zero set of $\sin$ is not a finite union of intervals), but it is definably complete (because every bounded subset of $\R$ has a supremum, be it definable or not). 

Many classical results of the differential calculus extend to definably complete expansions of an ordered field, for instance we have the following form of the Taylor formula with Lagrange form of the remainder. 

\begin{prop}[{\cite[Theorem 5]{Servi2008}}] \label{prop:Taylor} Let $M$ be a definably complete expansion of an ordered field and let $f:(a,b) \to M$ be a definable function of class $C^n$. Then for every $x, x_0\in (a,b)$ 
\[
f(x)
= \sum_{k=0}^{n-1} \frac{f^{(k)}(x_0)}{k!}\,(x - x_0)^k
  + \frac{f^{(n)}(\xi)}{n!}\,(x - x_0)^n
\]
for some $\xi$ in the interval between $x$ and $x_0$. 
\end{prop}

\section{Exponential fields}\label{sec:exp-fields}

Let $\R$ be the field of real numbers. We consider $\R$ as a first-order structure in the language $L = \{\leq,0,1,+,\cdot\}$ (with the natural interpretation of the symbols) and we use the same notation for $\R$ and its domain. 

\begin{defn}
    Let $\Rexp = (\R,\exp)$ be the expansion of $\R$ with the exponential function $\exp(x) = e^x$ and let $\Rrexp = (\R, \texp)$ be the expansion of $\R$ with the function $\texp:\R\to \R$ defined by 

\[
\texp(x)=
\begin{cases}
\exp(x) & \text{for } x\in (-1,1),\\
0       & \text{for } x\nin (-1,1).
\end{cases}
\]
\comment{Changed $\in$ with $\nin$ in the definition of $\texp$.}
We call $\texp$ the restricted exponential function. We consider $\Rexp$ and $\Rrexp$ as structures in the languages $ \{\leq, 0,1,+,\cdot, \exp \}$ and  $ \{\leq, 0,1,+,\cdot, \texp \}$ respectively. Let $T_{\exp}$ be the complete theory of $\Rexp$ and $T_{\texp}$ be the complete theory of $\Rrexp$. 
\end{defn}

We shall consider recursive subtheories $\Tglob \subseteq T_{\exp}$ and $\Tres\subseteq T_{\texp}$ axiomatized by the scheme of definable completeness plus the natural differential equations for $\exp$ and $\texp$ respectively. We need a couple of preliminary observations. 

\begin{prop}\label{prop:Fratarcangeli}
    Let $M$ be a definably complete expansion of an ordered field and let $f:M\to M$ be a definable function. Suppose that $f'(x) = f(x)$ for all $x$ in an open interval $(-a,a)$ and that $f(0)=1$. Then 
    $$f(x+y) = f(x)f(y)$$ whenever $x,y,x+y\in (-a,a)$. 
\end{prop}

\begin{proof} 
For $y \in (-a,a)$, consider the definable function $g_y$ defined on $I_y = (-a,a) \cap (-a-y,a-y)$ by $g_y(x) := f(x+y)/f(y)$. Note that if $x$ and $y$ are as in the statement then $x$ lies in the domain of $g_y$. An easy computation shows that $g_y' = g_y$ on $I_y$, so $f$ and $g_y$ are two solutions on $I_y$ of the differential equation $Y' = Y$, both taking the value $1$ at $x=0$. 
By \cite[Proposition 2.7]{Fratarcangeli2008} (based on \cite[Theorem 2.3]{Otero1996}) it follows that $g_y = f$ on $I_y$, hence $f(x+y)/f(y) = g_y(x) = f(x)$ when $x,y,x+y\in (-a,a)$ and the desired result follows.  
\end{proof}

The following converse also holds. 

\begin{prop}\label{prop:functional-to-differential}
    Let $M$ be a definably complete expansion of an ordered field and let $f:M\to M$ be a definable monotone function. Suppose that $f(0) = 1$ and $f(x+y) = f(x)f(y)$ whenever $x,y,x+y\in (-a,a)$. Then $f$ is differentiable on $(-a,a)$ and $$f'(x) = f'(0)f(x)$$ for all $x\in (-a,a)$. 
\end{prop}

\begin{proof}
   Every definable monotone function $f:M\to M$ in a definably complete expansion $M$ of a ordered field is continuous almost everywhere \cite[Lemma 6.6]{FornasieroHieronymi2015} and it is differentiable on a dense subset \cite[Theorem B]{FornasieroHieronymi2015}. If $f:M\to M$ satisfies our hypothesis, continuity at one point implies continuity everywhere, and similarly for differentiability. Granted this, we have
   $$f'(x) = \lim_{\eps \to 
   0} \frac{f(x+\eps)-f(x)}{\eps}=\lim_{\eps \to 
   0} f(x)\frac{f(\eps)-f(0)}{\eps} = f'(0) f(x).$$
\end{proof}

\begin{defn}
    Given a field of characteristic zero $F$, an {\em exponential function} is a morphism $E:F\to F^\times$ from the additive to the multiplicative group of the field. We then call $(F,E)$ an {\em exponential field}. 
\end{defn}

\begin{defn}
    Let $\Tglob \subseteq T_{\exp}$ be the theory in the language $\{\leq,0,1,+,\cdot, \exp\}$ whose models are the definably complete expansions of an ordered field with a differentiable function $\exp$ satisfying the following: 
    \begin{itemize}
        \item $\exp'(x) = \exp(x)$ for all $x$;
        \item $\exp(0) = 1$.
    \end{itemize}
\end{defn}

By Proposition \ref{prop:Fratarcangeli}, a model of $\Tglob$ satisfies $\exp(x+y) = \exp(x)\exp(y)$ for all $x,y$, so $\exp$ is an exponential function. Moreover, any model of $\Tglob$ satisfies $\exp' = \exp$, so the models of $\Tglob$ coincide with the EXP-fields in the sense of \cite{Krapp2023}. 

\begin{rem}
    By \cite[Theorem 14]{Dahn1983}, if $M$ is an exponential field satisfying $\exp(x) \geq 1+x$ for every $x$, then $\exp$ is continuous and differentiable and satisfies $\exp' = \exp$, so $\Tglob$ can be alternatively axiomatized using the above inequality.  
\end{rem}

\begin{defn}
    A \emph{restricted exponential field} is an ordered field $M$ equipped with a function  $\texp:M \to M$ satisfying $\texp(0) = 1$ and $\texp(x+y) = \texp(x)\texp(y)$ whenever $x,y,x+y \in (-1,1)$ and $\texp(x) = 0$ for $x\nin (-1,1)$. 
We call $\texp$ a \emph{restricted exponential function}. 
\end{defn}

\begin{defn}
    Let $\Tres \subseteq T_{\texp}$ be the theory in the language $\{\leq,0,1,+,\cdot, \texp\}$ whose models are the definably complete structures $(M,\leq,0,1,+,\cdot, \texp)$ where 
    $(M,\leq,0,1,+,\cdot)$ is an ordered field and $\texp:M\to M$ is a function that is differentiable on $(-1,1)$ and satisfies:
    \begin{itemize}
        \item $\texp'(x) = \texp(x)$ for $x\in (-1,1)$;
        \item $\texp(0) = 1$;
        \item $\texp(x) = 0$ for $x\nin (-1,1)$.
    \end{itemize}
\end{defn}

By Proposition \ref{prop:Fratarcangeli}, a model of $\Tres$ satisfies $\texp(x+y) = \texp(x)\texp(y)$ whenever $x,y,x+y \in (-1,1)$, so it is a restricted exponential field.

\begin{thm}[\cite{FornasieroServi2010,Hieronymi2011}]\label{thm:FSH} We have:
\begin{itemize}
\item Every model of $\Tglob$ is o-minimal. 
\item Every model of $\Tres$ is o-minimal.  
\end{itemize}
\end{thm}
\begin{pf}
    This is a special case of a more general result concerning Pfaffian functions, of which exponentiation is an example. More precisely, in \cite{FornasieroServi2010} the authors prove the o-minimality of every definably complete expansion of an ordered field with Pfaffian functions assuming a definable version of the Baire property. In \cite{Hieronymi2011} it is shown that the Baire property already follows from definable completeness. 
\end{pf}

A result of Ressayre in \cite{Ressayre1993} shows that $T_{\exp}$ is recursively axiomatized modulo $T_{\texp}$ (another proof appears in \cite[(4.10)]{Dries1994}). The following minor variant shows that $T_{\exp}$ is in fact finitely axiomatized modulo $T_{\texp}$. 

\begin{prop} \label{prop:ressayre}
     Let $R$ be an ordered field and let $E:R\to R$ be a function on $R$. 
     Then $(R,E)$ is elementarily equivalent to $\R_{\exp}$ if and only if the following hold:
     \begin{enumerate}
         \item $E:R\to R^{>0}$ be an increasing surjective exponential function.
         \item $E(x) \geq x+1$ for all $x$. 
  \item    $(R,E_{|(-1,1)})$ is elementarily equivalent to $\R_{\texp}$.
     \end{enumerate}
\end{prop} 

\begin{pf} 
The result appears in \cite{Ressayre1993} but with axiom $(2)$ replaced by an axiom scheme $(2')$ whose $n$-th instance $(2')_n$ says that $E(x) > x^n$ for every $x > f(n)$ where $f(n)$ is a recursive function of $n$.  It therefore suffices to show that, in the presence of the other axioms, $(2)$ implies $(2')_n$. For simplicity we take $f(n) = 4n^2$ (in \cite{Dries1994} the authors take $f(n) = n^2$, but this is irrelevant). So assume $x> 4n^2$. Then $\sqrt x > 2n$ and we obtain
$E(x) > E(2n \sqrt x) = E(\sqrt x)^{2n} \geq (\sqrt x +1)^{2n} \geq x^{n}$.
\end{pf}

\begin{prop}\label{prop:growth}
    Let $M\models \Tglob$. Then $M$ satisfies conditions (1) and (2) in Proposition \ref{prop:ressayre}.
\end{prop}
\begin{pf} Condition (1) clearly holds. For (2) observe that, by the Taylor formula, $\exp(x) = 1+x+\exp(\xi)\tfrac {x^2} 2$ for some $\xi\in (0,x)$, so $\exp(x) \geq x+1$ for all $x$. 
\end{pf}

Besides exponential functions and restricted exponential functions, we also need the notion of partial exponential function defined as follows. 

\begin{defn}[\cite{Kirby2010}] \label{defn:exponential field}
Given a field $F = (F,0,1,+,\cdot)$ of characteristic $0$, a {\em partial exponential function} is a homomorphism $\exp:(D(F),+) \to (F,\cdot)$ where $D(F) = \dom(\exp)$ is a $\mathbb{Q}$-vector subspace of $F$.  We then say that $(F, \exp) =(F, 0,1,+,\cdot, \exp)$ is a {\em partial exponential field}. 
\end{defn}

\begin{defn}
    Let $M$ be an expansion of an ordered field. An element $x \in M$ is \emph{finite} if there is $N \in \N$ such that $x \in (-N,N)$. It is \emph{infinitesimal} if for every positive $N \in \N$ we have $x \in \left(-\frac{1}{N}, \frac 1N\right)$.
\end{defn}  

\begin{prop} \label{prop:exp-finite}
Given an ordered field $M$ with a restricted exponential function $\texp:M\to M$, there is unique partial exponential function 
$$\Exp:\Fin(M)\to \Fin(M)$$
that agrees with $\texp$ on $(-1,1)$. Moreover, for every $n\in \N$, the restriction of $\Exp$ to any interval $(-n,n)$ is definable in $(M, \texp)$. 
\end{prop}
\begin{proof}
 Given $x\in \Fin(M)$, choose $n\in \N$ such that $x \in (-n,n)$ and define $\Exp(x) = \texp(x/n)^n$. 
 Using the functional equation $\texp(x+y) = \texp(x)\texp(y)$ for $x,y\in (-1,1)$, it is easy to see that the definition does not depend on the choice of $n$ (e.g. $\texp(x/n)^n = \texp(x/nk)^{nk} = \texp(x/k)^k$).
\end{proof}

\begin{prop}\label{prop:exp-finite-2}
    If $(M,\texp)\models \Tres$ and $\Exp:\Fin(M)\to \Fin(M)$ is the induced partial exponential function, then
    \begin{enumerate}
        \item $\Exp$ is smooth on $\Fin(M)$, 
        \item $\Exp' = \Exp$ on $\Fin(M)$. 
    \end{enumerate}
\end{prop}
\begin{pf}
    In models of $\Tres$ the function $\texp$ is smooth on $(-1,1)$ and satisfies $\texp(0) = 1$ and $\texp'(x) = \texp(x)$ for $x\in (-1,1)$. Let $n\ge 1$ be an integer and, for $x\in(-n,n)$, define $h(x)=\Exp(x/n) = \texp(x/n)$. Then $\Exp(x)=h(x)^n$, and by the chain rule (which holds in definably complete expansions of fields),
$\Exp'(x)=n\,h(x)^{n-1}h'(x)$.
Moreover,
$h'(x)=\tfrac1n\,h(x)$, so 
$\Exp'(x)=h(x)^n=\Exp(x)$
for every $x\in(-n,n)$.
Finally observe that $\Exp'(x) = \Exp'(0)\Exp(x)$ by Proposition \ref{prop:Fratarcangeli} (applied to the restriction of $\Exp$ to $(-n,n)$) and $\Exp'(0) = \texp'(0) = 1$. 
\end{pf}

\section{Convex valuation rings}\label{sec:conv-val-rings}
Krapp showed that if $M\models \Tglob$ and $\mu$ is the maximal ideal of $\Fin(M)$, then the residue field $\Fin(M)/\mu$ has an induced exponential function \cite[Proposition 3.2]{Krapp2019} and a unique embedding into $\R_{\texp}$ as an exponential field. The argument uses some compatibility results between exponential functions and convex valuations studied by Salma Kuhlmann, see in particular \cite[Lemma 1.17]{Kuhlmann2000}. Here we adapt some of the arguments of Krapp and S. Kuhlmann to models of $\Tres$. 

\begin{lem}\label{lem:exp(-1)}
    Let $M \models \Tres$ and let $\Exp:\Fin(M)\to \Fin(M)$ be the induced partial exponential function. For $n \in \N$ odd we have:
    \[\sum_{k=0}^n \frac{(-1)^k}{k!} < \Exp(-1) < \sum_{k=0}^{n+1} \frac{(-1)^k}{k!}\]
\end{lem}
\begin{proof} Since the restriction of $\Exp$ to $(-2,2)$ is definable, 
we can use the Taylor formula in $M$ to obtain $\Exp(-1) = \sum_{k=0}^{n-1} \frac{(-1)^k}{k!} + \Exp(\xi)\frac{(-1)^n}{n!}$ for some $\xi \in (-1,0)$. Since $\Exp(\xi) = \Exp(\xi/2)^2$ is positive, it follows that if $n$ is even $\Exp(-1)>\sum_{k=0}^{n-1} \frac{(-1)^k}{k!}$ and if $n$ is odd $\Exp(-1) < \sum_{k=0}^{n-1} \frac{(-1)^k}{k!}$.
\end{proof}

\begin{prop}\label{prop:induced-texp}
Let $M\models \Tres$ and let $\mathcal{O}\subset M$ be a convex subring with maximal ideal $\mathfrak{m} \subset \mathcal{O}$. Then $\texp((-1,1)) \subset \mathcal{O}^\times$ and $\texp(\mathfrak{m}) \subseteq 1 + \mathfrak{m}$. 
\end{prop}
\begin{pf}
By Lemma \ref{lem:exp(-1)}, for every $q \in \Q$ we have $\Exp(q)=\Exp(-1)^{-q} \in \Fin(M)^\times$. Since $\Exp$ is increasing, $\Exp$ maps $\Fin(M)$ into $\Fin(M)^\times$, hence $\texp$ maps $(-1,1)$ into $\Fin(M)^\times$, and the latter is contained in $\mathcal{O}^\times$. 
By the Taylor formula (Proposition \ref{prop:Taylor}), for $\varepsilon \in \mathfrak{m}$ we have $\texp(\varepsilon)=1+\varepsilon+\exp(\xi) \frac{\eps^2}{2}$ for some $\xi$ between $0$ and $\eps$. Since $\exp(\xi)\in \mathcal O$, it follows that $\exp(\xi) \frac{\eps^2}{2}\in \mathfrak m$, hence $\texp(\varepsilon) \in 1+\mathfrak{m}$.
\end{pf}

\begin{cor}\label{cor:induced-texp}
    Let $M\models \Tres$ and let $\mathcal{O}\subset M$ be a convex subring with maximal ideal $\mathfrak{m} \subset \mathcal{O}$.
    Then the residue field $\mathcal{O}/\mathfrak{m}$ has an induced restricted exponential function $\overline{\texp}:\mathcal O/\mathfrak{m}\to \mathcal O/\mathfrak{m}$ defined by 
$$
\overline{\texp}(x+\mathfrak{m}) = 
\begin{cases}
\texp(x) + \mathfrak{m} & \text{ if } x + \mathfrak{m}\subset (-1,1) \\ 
0 & \text{ if } x + \mathfrak{m} \not\subset (-1,1).
\end{cases}
$$ 
Moreover, if $\overline{\Exp}:\Fin(\mathcal{O}/\mathfrak{m})\to \Fin(\mathcal{O}/\mathfrak{m})$ is the induced partial exponential function (as in Proposition \ref{prop:exp-finite}), 
then $\overline{\Exp}(x+\mathfrak{m}) = \Exp(x)+\mathfrak{m}$. 
\end{cor}

\begin{pf}
    If $x+\mathfrak{m} \subset (-1,1)$ and $x-y \in \mathfrak{m}$, then by Proposition \ref{prop:induced-texp} we have $\frac{\texp(x)}{\texp(y)}=\texp(x-y) \in 1+\mathfrak{m}$. It follows that $\texp(x) \in \texp(y) + \mathfrak{m} \texp(y)$, and since $\texp(y) \in \mathcal{O}$ we have $\texp(x)-\texp(y) \in \mathfrak{m}$. So the definition of $\overline{\texp}$ is well-posed.

    Finally, if $x +\mathfrak{m} \in \Fin(\mathcal{O}/\mathfrak{m})$ then there is $n \in \N$ such that $\frac{x+\mathfrak{m}}{n} \subset (-1,1)$ and \[\overline{\Exp}(x+\mathfrak{m})=\overline{\texp} \left( \frac{x}{n}+\mathfrak{m}\right)^n =\left(\texp\left( \frac xn\right)+\mathfrak{m}\right)^n=\texp\left(\frac xn \right)^n+\mathfrak{m}=\Exp(x)+\mathfrak{m}. \]
\end{pf}

We do not know whether the restricted exponential field $(\mathcal O/\mathfrak{m}, \overline{\texp})$ defined above is a model of $\Tres$, since a priori it may not be definably complete.

Consider now the case $\mathcal{O} = \Fin(M)$. In this case the residue field is archimedean and we obtain the following result. 

\begin{prop}\label{prop:archimedean-residue}
    Let $M \models \Tres$, and let $\mu$ denote the maximal ideal of $\Fin(M)$. Then $\Fin(M)/\mu$ has a unique embedding into $\Rrexp$ as a restricted exponential field. 
\end{prop}

\begin{pf}
Since $\Fin(M)/\mu$ is archimedean, it has a unique ordered field embedding $\iota: \Fin(M)/\mu \to \R$ into the reals \cite{KaplanskyMaximal1942}. We need to show that the embedding  preserves the restricted exponential function. Let $\overline{M}:=\img(\iota) \subseteq \R$. For $x\in \Fin(M)$, let $\overline x = \iota(x+\mu) \in \overline{M}$. Since $\Fin(\overline{M})=\overline{M}$, by Corollary \ref{cor:induced-texp} there is a well-defined exponential function $\overline{\Exp}:\overline{M}\to \overline{M}$ given by $\overline{\Exp}(\overline x) =\overline{\Exp(x)}$.

It follows from Lemma \ref{lem:exp(-1)} that $\overline{\Exp}(-1)=e^{-1}$, and therefore for every $q \in \Q$ we obtain $\overline{\Exp}(q)=\overline{\Exp}(-1)^{-q}=e^{-q}$, so $\overline{\Exp}$ coincides with the real exponential on $\Q$ and by continuity $\overline{\Exp}(x) = e^x$ for every $x \in \overline M$. Considering the restriction of $\overline{\Exp}$ to $(-1,1)$ the desired result follows. 
\end{pf}

In Theorem \ref{thm:embedding} we will show that, under Schanuel's conjecture, $\Fin(M)/\mu$ embeds in $M$ as a restricted exponential field.

\section{Tangent spaces and regular points}\label{sec:tan-spaces}

We need a few results about differentiable definable functions in the o-minimal setting, in particular regarding regular points and transversality, see for instance \cite{Berarducci2001,Servi2008, ServiTesi2008}. 

Let $N$ be an o-minimal expansion of an ordered field, for instance a model of $\Tres$ or $\Tglob$. 
Let $0<p\in \N$ and let $f:X\to Y$ be a definable map between two definable sets  $X\subseteq N^n$ and $Y\subseteq N^k$. We say that $f$
is a {\em definable $C^p$-map} around $x\in X$ if there is a 
definable open neighborbood $U\subset N^n$ of $x$ and a definable map
$F\colon U\to N^k$ extending $f_{|U \cap X}$ such that all the
partial derivatives $\partial^s F/ \partial x_{i_1}\cdots \partial
x_{i_s}$ exist and are continuous for each $s\leq p$. If this happens around every point of $X$, we say that $f$ is a definable $C^p$-map.  
A {\em definable $C^p$-diffeomorphism} is a bijective definable $C^p$-map whose inverse is a definable $C^p$-map. 

Let $X$ be a definable subset of $N^n$. Given $x\in X$ we define $T_x(X) \subseteq N^n$ as the set of all $v\in N^n$ such that there is a definable curve $\gamma:(-\eps,\eps)\to X$ of class $C^1$ with $\gamma(0) = x$ and $\gamma'(0) = v$, where $\gamma'(0) = \lim_{t\to 0}\frac{\gamma(t)-\gamma(0)}{t}$. Without additional hypotheses, the sum of two vectors of $T_x(X)$ may not lie in $T_x(X)$. However,  
if $X \subseteq N^n$ is locally definably $C^1$-diffeomorphic to an open subset of $N^d$ around $x\in X$, then $T_x(X)$ is an $N$-linear subspace of $N^n = T_x(N^n)$ of dimension $d$. In this case we say that $X$ is a {\em definable $C^1$-manifold} of dimension $d$ around $p$ and $T_x(X)$ is its {\em tangent space at $x$}. The notion of definable $C^p$-manifold is defined similarly. 

If $f:X\to Y$ is a definable $C^1$-map around $x\in X$ and $v = \gamma'(0)\in T_x(X)$ (where $\gamma:(-\eps,\eps)\to X$ is a definable function of class $C^1$ with $\gamma(0) = x$), then $df_x:T_x(X)\to T_{f(x)}(Y)$ is defined by $df_x(v) = (f \circ \gamma)'(0)$. If $X$ is a definable $C^1$-manifold around $x$ and $Y$ is a definable $C^1$-manifold around $f(x)$, then $df_x$ is a linear map. 

In particular, if $f = (f_1, \ldots, f_k):N^n\to N^k$ is a definable function of class $C^1$ around $x\in N^n$, then $df_x:N^n\to N^k$ is the linear map given by the Jacobian matrix $J(f_1, \ldots, f_k)(x) = (\partial f_i/\partial x_j)_{ij}(x)$.  

\begin{defn}
    Given an open definable set $U \subseteq N^n$  and definable functions $f_1,\dots,f_k:U \to N$, we write $V_U(f_1,\dots,f_k)$ for the set of common zeros of $f_1,\dots,f_k$ in $U$. When $U=N^n$, we write $V_n(f_1,\dots,f_k)$. 
    
    If $f = (f_1, \ldots, f_k)$ is of class $C^1$ around a point $x \in V_U(f_1,\dots,f_k)$, and $df_x$ has maximal rank (equal to $k$), we say that $x$ is a {\em regular point} of $f$. We write $V_U^{\reg}(f_1, \ldots, f_k) \subseteq V_U(f_1, \ldots, f_k)$ for the set of regular points. 
\end{defn}

If $X = V_U(f_1, \ldots, f_k) \subseteq N^n$ and $f_1,\dots,f_k$ are definable $C^1$-maps around $x \in X$, then $T_x(X) \subseteq \ker (df_x)$. The equality $T_x(X) = \ker (df_x)$ holds if and only if $x$ is a regular point of $f$. In this case $T_x(X)$ has dimension $n-k$ and $X$ is locally definably $C^1$-diffeomorphic to an open subset of $N^{n-k}$ around $x$. We thus have:   
\begin{prop}[{\cite[Theorem 3.2]{Berarducci2001}}]\label{prop:dimension}
    If $X = V_U^{\reg}(f_1, \ldots, f_k) \subseteq N^n$ is non-empty, then it has dimension $n-k$. In particular, if $n=k$, then $\dim(X) = 0$, so $X$ is a finite set by o-minimality.
\end{prop}

\begin{rem}\label{rem:meet-transversally} Let $U$ be an open subset of $N^n$. 
Let $f:U\to N^k$ and $g:U \to N^m$ be definable $C^1$-maps around $x$ with sets of zeros $X = V_U(f)$ and $Y= V_U(g)$, and let $x \in X \cap Y$ be a regular point of both $f$ and $g$. Then $x$ is a regular point of $(f,g):U \to N^{k+m}$ if and only if $X$ and $Y$ meet {\em transversally} at $x$, i.e. $T_x(X)+T_x(Y) = N^n$. 
\end{rem}

\section{Restricted exponential polynomials}\label{sec:res-exp-pol}

\begin{defn} 
    Let $N$ be a restricted exponential field, $R \subseteq N$ a subring. An {\em $(\ell,n)$-$\E$-polynomial} over $R$ is an expression of the form $$F_p(x_1, \ldots, x_n) := p(x_1, \ldots, x_n, \E(x_1), \ldots, \E(x_\ell))$$
    where $\ell \leq n$ and $p(x_1, \ldots, x_n, y_1, \ldots, y_\ell) \in R[\bar x, \bar y]$ is a polynomial over $R$.
\end{defn}

An $(\ell,n)$-$\E$-polynomial defines a function $F_p$ from $N^n$ to $N$ which is possibly discontinuous (since $\texp$ is discontinuous at $-1,1$). If $N\models \Tres$, the restriction of $F_p$ to $(-1,1)^\ell\times N^{n-\ell}$ is continuous and smooth.  

\begin{defn} 
    We write $U_{\ell,n}$ for the set $(-1,1)^\ell \times N^{n-\ell} \subseteq N^n$ and $V_{\ell,n}(f_1,\dots,f_k)$ for $V_{U_{\ell,n}}(f_1,\dots,f_k)$, namely the set of common zeros of $f_1,\dots,f_k$ in $U_{\ell,n}$. Similarly we define $V_{\ell,n}^{\reg}(f_1,\dots,f_k)=V_{U_{\ell,n}}^{\reg}(f_1,\dots,f_k)$. 
\end{defn}

\begin{defn} Let $N$ be a restricted exponential field, $R \subseteq N$ a subring.
\begin{itemize}
    \item An {\em $(\ell,n)$-$\E$-variety} over $R$ is a set of the form $V_{\ell,n}(f_1, \ldots, f_k)$ where $f_1,\dots,f_k$ are $(\ell,n)$-$\E$-polynomials over $R$. 
    \item An {\em $(\ell,n)$-Khovanskii point} over $R$ is an $n$-tuple 
    $$(\alpha, \beta) = (\alpha_1, \ldots, \alpha_\ell, \beta_1, \ldots, \beta_{n-\ell})\in U_{\ell,n} \subseteq N^n$$ 
    that belongs to a set of the form $V_{\ell,n}^{\reg}(f_1, \ldots, f_n)$ where $f_1, \ldots, f_n$ are $(\ell,n)$-$\E$-polynomials over $R$. 
    \item An {\em $(\ell,n)$-Khovanskii system} over $R$ is a system of equations and inequations of the form
	\[
	\bigwedge_{i=1}^n f_i(x_1, \ldots, x_n)=0 \; \wedge \; \det(J(f_1, \ldots, f_n)(x_1, \ldots, x_n))\neq 0
	\] 
	where $f_i=F_{p_i}$ are $(\ell,n)$-$\E$-polynomials over $R$ and $J(f_1, \dots, f_n) = \left(\frac{\partial f_i}{\partial x_j}\right)_{ij}$ is the Jacobian matrix of $f_1, \ldots, f_n$ defined formally using the rule $\frac{\partial\texp(x)}{\partial x}=\texp(x)$.
\end{itemize}
Note that these definitions assume a given ordering of the variables, namely the first $\ell$ variables are restricted to $(-1,1)$. 
When there is no need to specify $n,\ell,R$, we simply speak of $\E$-polynomial, $\E$-variety, Khovanskii point, Khovanskii system. 
\end{defn}

It follows from the definitions that the $(\ell,n)$-Khovanskii points are the points that satisfy some $(\ell,n)$-Khovanskii system. 

\begin{defn}
We say an $\E$-polynomial has \emph{complexity} $\leq \ell$ if it is an $(\ell,n)$-$\E$-polynomial for some $n \in \N$. We similarly define the complexity of an $\E$-variety, Khovanskii point, or Khovanskii system.
\end{defn}

\begin{prop}\label{prop:servi}
Let $N \models \Tres$. Every non-empty $\E$-variety of complexity $\leq \ell$ contains a projection of an $\E$-Khovanskii point of complexity $\leq \ell$ defined over the same parameters. 
 \end{prop}

Similar results can be found in \cite[Section 4]{wilkieTheoryRealExponential1989} and \cite[Theorem 5.1]{Wilkie1996}, but they are stated under the assumption that the ambient structure is elementarily equivalent to an expansion of the reals. The proposition can however be proved by induction on the dimension using the next lemma. We say that an $\texp$-variety $V_{\ell,n}(f_1,\dots,f_k)$ is \emph{regular} if it coincides with $V_{\ell,n}^{\reg}(f_1,\dots,f_k)$, and we say a definable set $Z$  is a  \emph{regular component} of complexity $\leq \ell$ if there is an $(\ell,n)$-$\E$-variety $V_{\ell,n}(f_1,\dots,f_k)$ such that $Z$ is a clopen subset of $V_{\ell,n}^{\reg}(f_1,\dots,f_k)$. We then call $Z$ a regular component of the variety $V_{\ell,n}(f_1,\dots,f_k)$. These notions depend not only on the variety, but on its presentation as the set of zeros of a given system. 

\begin{lem}[{\cite[Theorem 33, Remark 42, Theorem 44]{Servi2008}}] \label{lem:servi3344}
    Let $N \models \Tres$, and let $V=V_{\ell,n}(f_1,\dots,f_k)$ be an $(\ell,n)$-$\E$-variety over a subfield $M \subseteq N$. 
    \begin{enumerate}
        \item $V$ can be written as a finite union of regular components of $(\ell,n)$-$\texp$-varieties over $M$. \label{33441}
    
        \item The set of regular points of $V$ is the projection of a finite union of regular $\texp$-varieties over $M$ of complexity $\leq \ell$ and dimension $\leq \dim V$.\label{33440}
        
        \item If $V$ is regular and $\dim(V)>0$, there are $\E$-polynomials $h_1,\dots,h_m$ over $M$ of complexity $\leq \ell$ such that for every non-empty clopen definable subset $S \subseteq V$ there is $i \in \{1,\dots,m\}$ such that $S \cap V^{\reg}_{n,\ell}(f_1,\dots,f_k,h_i) \neq \emptyset$. \label{33442}
    \end{enumerate}
\end{lem}

Lemma \ref{lem:servi3344} is a special case of results in \cite{Servi2008}, observing that the ring of $(\ell,n)$-$\E$-polynomials is Noetherian and closed under differentiation and that the $0$-dimensional $(\ell,n)$-$\E$-varieties are finite by o-minimality.
The only minor issue is that the results of \cite{Servi2008} apply to smooth functions on a field, while our $\E$ is discontinuous at $-1$ and $1$. To handle this we can compose all the functions with the semialgebraic diffeomorphism from $M^n$ to $(-1,1)^\ell \times M^{n-\ell}$ given by $x_i\mapsto x_i/\sqrt{1+x_i^2}$ for $i=1, \ldots, \ell$.    

 \begin{proof}[Proof of Proposition \ref{prop:servi}]
  By Lemma \ref{lem:servi3344}(\ref{33441}) we can write $V$ as the union of finitely many regular components of $(\ell,n)$-$\E$-varieties, so it suffices to show that every regular component $S$ of positive dimension contains the projection of a regular component of strictly smaller dimension defined over the same parameters and of the same complexity. To this aim, write $S$ as a clopen subset of the set of regular points of some $(\ell,n)$-$\texp$-variety $V(G)$, where $G=(g_1,\dots,g_k)$. By Lemma \ref{lem:servi3344}(\ref{33440}), $V^{\reg}(G)$ is the projection of a finite union of regular $\texp$-varieties of complexity $\leq \ell$. One of these regular varieties, say $V_{\ell,m}(H)=V_{\ell,m}^{\reg}(H)$, must then contain a definably connected component $C$ whose projection is contained in  $S$. By Lemma \ref{lem:servi3344}(\ref{33442}), $C$ intersects a regular component of an $\texp$-variety of the same complexity and lower dimension, defined over the same parameters.
 \end{proof}

\begin{defn}\label{def:gexpnl} Let $N\models \Tres$, $n,\ell \in \N$. Define:
$$G_{\exp}^{n,\ell}:=\left\{(x_1,\dots, x_n, y_1, \ldots, y_\ell) \in (-1,1)^\ell \times N^{n-\ell} \times (N^\times)^\ell \;\middle| \; \bigwedge_{i=1}^\ell y_i = \texp(x_i) \right\}$$
$$G_{\exp}^{n,\ell,-}:=\left\{(x_1,\dots, x_n, y_1, \ldots, y_\ell) \in (-1,1)^\ell \times N^{n-\ell} \times (N^\times)^\ell \;\middle| \; \bigwedge_{i=1}^{\ell-1} y_i = \texp(x_i) \right\}$$
Note that $\dim(G_{\exp}^{n,\ell}) = n$ and $\dim(G_{\exp}^{n,\ell,-}) = n+1$.
\end{defn}

\begin{lem}\label{lem:exists-variety} Let \replaced{$N\models \Tres$ and let $M\subseteq N$ be a subfield. Let }{$M \subset N$ be models of $\Tres$ and let } $(\alpha,\beta) \in N^n$ be an $(\ell,n)$-Khovanskii point over $M$. Then there is an algebraic variety $V$ defined over $M$ of dimension $\ell$ such that $V$ and $G_{\exp}^{n,\ell}$ meet transversally at $(\alpha,\beta,\texp(\alpha))\in N^{n+\ell}$. In particular, the transcendence degree $\td(\alpha,\beta, \texp(\alpha)/M)$ is at most $\ell$.
\end{lem}

\begin{proof}
     If the $\texp$-Khovanskii point $(\alpha,\beta)$ is a regular solution of the system of equations 
$$\bigwedge_{i=1}^n F_{p_i}(x_1,\dots,x_n) = 0$$
  consider the variety $V=V_{n+\ell}(p_1,\dots,p_n)$ defined by the polynomials $p_i(\bar x, \bar y)$. Observe that $(\alpha,\beta, \texp(\alpha))$ is a regular solution of the system consisting of the algebraic equations $$p_i(x_1, \ldots, x_n, y_1, \ldots, y_\ell) = 0$$ ($i=1, \ldots, n$) together with the equations $\texp(x_i) = y_i$ ($i = 1, \ldots, \ell$), hence $V$ and $G_{\exp}^{n,\ell}$ meet transversally at $(\alpha,\beta,\texp(\alpha))$ by Remark \ref{rem:meet-transversally}. A priori, $V$ may have dimension greater than $\ell$, but it has dimension $\ell$ around the point $(\alpha,\beta,\texp(\alpha))$, so there is an open semialgebraic subset $U \subset V$ of dimension $\ell$ which contains $(\alpha,\beta,\texp(\alpha))$. By \cite[Proposition 2.8.2]{bochnak2013real} the Zariski closure $U^{\text{Zar}} \subseteq V$ of $U$ has dimension $\ell$, so after possibly replacing $V$ with $U^{\text{Zar}}$ we conclude.
\end{proof}

\section{Exponential algebraicity}\label{sec:exp-alg}
The notions of $\E$-polynomial, $\E$-Khovanskii point, and Khovanskii system have obvious analogues in the context of partial exponential fields (Definition \ref{defn:exponential field}), which have been thoroughly studied for example in \cite{Kirby2010}. The notion of \emph{exponential algebraicity} from \cite{Kirby2010} naturally adapts to our context as well.

\begin{defn}
    Let $M$ be a restricted exponential field, $C \subseteq M$. We say a point $a \in M$ is $\texp$-algebraic over $C$ if $a$ is a coordinate of an $\E$-Khovanskii point over the restricted exponential subfield of $M$ generated by $C$.

    We write $\ecl^M(C)$ for the set of points in $M$ which are $\texp$-algebraic over $C$. A subset $C \subseteq M$ is $\ecl$-closed if $\ecl^M(C)=C$.
\end{defn}

The use of the notation $\ecl$, borrowed from \cite{Kirby2010}, is justified by the following proposition. Note that definable completeness is not assumed. 

\begin{prop}\label{prop:pregeometry}    
    Let $M$ be a restricted exponential field and let $\Exp$ be as in Proposition \ref{prop:exp-finite}. The operator $\ecl^M$ coincides with exponential-algebraic closure in the partial exponential field $(M,\Exp)$ in the sense of \cite{Kirby2010}. In particular, it is a pregeometry.
\end{prop}

\begin{proof} 
    The fact that every solution of a Khovanskii system in the partial exponential field $(M,\Exp)$ is also a solution of a Khovanskii system in the restricted exponential field $M$ follows easily from the definition of $\Exp$. The fact that $\ecl^M$ is a pregeometry is then \cite[Theorem 1.1]{Kirby2010}.
\end{proof}

\begin{defn} 
For any subsets $X,Y$ of a restricted exponential field $\M$, 
\begin{itemize}
\item 
$\etd(X/Y)$ is the maximal cardinality of a subset of $X$ that is $\ecl$-independent over $Y$.  
\item $\td(X/Y)$ is the transcendence degree of the field extension
$\Q(X,Y)/\Q(X)$
\item $\ldim_\Q(X/Y)$ is the dimension of $\spn_{\Q}(X,Y)/\spn_{\Q}(Y)$
as a $\Q$-vector space.
\end{itemize} 
\end{defn}

\begin{thm}[Predimension inequality] \label{thm:kirby} Let $C$ be a $\ecl$-closed subset of a restricted exponential field $\M$. Then, \added{for every $\bar x$ in $(-1,1)$, we have}
		\[ \td(\bar x, \texp(\bar x)/C ) - \ldim_\Q (\bar x / C) \geq \etd(\bar x /C). \]
\end{thm}

\begin{proof}
    The corresponding result for partial exponential fields is \cite[Theorem 1.2]{Kirby2010}, so the theorem follows by Proposition \ref{prop:pregeometry}.
\end{proof}

\added{
Note that Theorem 1.2 in \cite{Kirby2010} is stated for simplicity for total exponential fields, but the proof goes through for partial exponential fields as well with no modifications. 
}

\section{Germs of polynomially bounded functions}\label{sec:germs}

Let $M$ be an expansion of an ordered field, and let $f:M\to M$ be a definable function. We recall that $f$ is said to be {\em polynomially bounded} if there exists $n\in \N$ such that $|f(x)| \leq x^n$ for all sufficiently large $x \in M$. 
If every definable function $f:M\to M$ is polynomially bounded, then the structure $M$ is called polynomially bounded. A theory $T$ is polynomially bounded if every model of $T$ is polynomially bounded. 
The theory $T_{an}$ studied in \cite{Denef1988,VandenDries1986,Dries1994} is polynomially bounded and complete, hence so is its reduct $T_{\texp}$. We cannot however deduce that the subtheory $\Tres \subseteq T_{\texp}$ is polynomially bounded without knowing whether it is complete. This justifies the following definition. 

\begin{defn} \label{defn:germs} Let $M\models \Tres$. 
\begin{itemize}
    \item Given a definable function $f:M\to M$, let $\germ(f)$ be the germ of $f$ at $+\infty$ and let $\mathcal G_M$ be the family of all such germs. Since $M$ is o-minimal, $\mathcal G_M$ is an ordered field (with the natural addition and multiplication of germs). 
\item 
Consider the convex subring $\mathcal{O}  \subseteq \mathcal G_M$ defined by 
$$
\mathcal{O} = \{\germ(f) \in \mathcal G_M \mid f \text{ is polynomially bounded} \}. 
$$
Let $\mathfrak m \subseteq \mathcal O$ be the maximal ideal and $R= \mathcal O/\mathfrak m$ the residue field. 
\item We denote by  
$$[f] = \germ(f)+\mathfrak m \in R$$
the class of $\germ(f)$ modulo $\mathfrak m$.
\end{itemize}
\end{defn}
Clearly if $M$ is polynomially bounded, then $\mathcal O = \mathcal G_M$ and $\mathfrak{m} = (0)$. 

\begin{rem} Let $M\models \Tres$ and let $R = \mathcal O/\mathfrak m$ be as above. 
\begin{itemize}
    \item 
$\germ(f) \in \mathfrak m$ if and only if $|f(t)|<t^{-n}$ for all $n\in \N$ and all sufficiently large $t\in M$, namely $f$ is {\em flat} at $+\infty$. 
\item $\mathcal G_M$ is an elementary extension of $M$ with a restricted exponential function defined by $\texp(\germ(f)) = \germ(\texp(f))$. 
\end{itemize}
For the second point, note that, since $M$ is o-minimal, $\mathcal G_M$ can be naturally identified with the definable closure of $M\cup \{t\}$ inside an elementary extension $N\succ M$ with an element $t>M$. The identification sends $\germ(f)\in \mathcal G_M$ to $f(t)\in N$.   
\end{rem}

\begin{prop}\label{prop:hopital}
Let $f,g:M \to M$ be definable functions not eventually constant with  $\lim_{t\to \infty}f(t) = \lim_{t\to \infty}g(t) = 0$. Then 
$$\lim_{t\to +\infty}f(t)/g(t) = 0 \iff \lim_{t\to +\infty}f'(t)/g'(t) = 0.$$
\end{prop}
\begin{proof} 
Consider the derivation $D:\mathcal G_M\to \mathcal G_M$ given by $D(\germ(f)) = \germ(f')$, where $f'= \frac {df(t)}{dt}$ (this makes sense since for every  definable function $f:M\to M$ the derivative $f'(t)$ exists for all sufficiently large $t$). The field of constants $\ker (D)$ can be identified with $M$. By o-minimality every bounded definable function is equal to a constant plus a definable function that tends to zero, hence $\mathcal G_M$ is an $H$-field in the sense of \cite{Aschenbrenner2002}. Now it suffices to observe that if $(k,D)$ is an $H$-field and $v$ is the valuation whose valuation ring is the convex hull of $\ker (D)$, then for any $a,b\in k^\times$ with $v(a), v(b) \neq 0$, we have $v(a) < v(b) \iff v(a') < v(b')$ (see \cite[Corollary 1]{Rosenlicht1980} and \cite[Lemma 1.1]{Aschenbrenner2002}).
\end{proof}

\begin{lem}\label{lem:germs} Let $R = \mathcal O/\mathfrak{m}$ be as in Definition \ref{defn:germs}. By Corollary \ref{cor:induced-texp}, $R$ inherits from $\mathcal G_M$ the structure of restricted exponential field. With this structure, $R$ satisfies:  
\begin{enumerate}
    \item There is a derivation $\delta: R \to R$ defined by $\delta [f] = \left[\frac {df(t)}{dt}\right]$ where $[f]\in R$ is the equivalence class modulo $\mathfrak m$ of the germ of $f$.  
    \item The field of constants $\ker (\delta)$ consists of the classes $[f]$ of the eventually constant functions $f:(0,+\infty)\to M$ and can be identified with $M$. 
    \item If $x \in (-1,1)_R$ and $y = \texp(x)$, then $\frac {\delta y} y = \delta x$. \label{diff-eq}
    \item $R$ is real closed. 
  \item If $S$ is a semialgebraic subset of $\mathcal G_M^n$ defined over $M$, then the image of $S \cap \mathcal{O}^n$ under the residue map is the topological closure of $S(R)$. \label{item:semi-closure}

\end{enumerate}
\end{lem}

\begin{proof}
For (1), we must show that if $\germ(f)\in \mathfrak m$, then $\germ(f')\in \mathfrak m$, so assume $\germ(f)\in \mathfrak m$. Given $n\in \N$, we must show that $f'(t)/t^{-n}$ tends to zero for $t$ tending to $+\infty$ in $M$. To this aim it suffices to apply Proposition \ref{prop:hopital} with $g = t^{-n+1}$. 

Point (2) is easy. 

For (3), suppose $y = \texp(x)$. Then there is $f$ such that $x = [f]$ and $y = [\texp(f)]$. The result then follows by the chain rule. 

\replaced{Point (4) is a special case of \cite[Lemma 5]{CherlinDickmann83}. We provide a short proof. }{For (4)} It suffices to show that every polynomial over $R$ that changes sign has a zero. So let $G(X)=\sum_{i=0}^d [g_i] X^i$ be a polynomial and suppose $G([a]) < 0 <  G([b])$. 
For each sufficiently large $t\in M$ we then have $\sum_{i=0}^d g_i(t)a(t)^i < 0 < \sum_{i=0}^d g_i(t)b(t)^i$. Since $M$ is real closed, the polynomial $\sum_{i=0}^d g_i(t)X^i$ has a root $X=c(t)$, where $c$ is a definable function of $t$ and $a(t) \leq  c(t) \leq b(t)$. Since $a,b$ are polynomially bounded, so is $c$, and $[c]$ is a root of $G(X)$ between $[a]$ and $[b]$. 

For (5), since $\mathcal{G}_M$ is real closed, there is an embedding $\iota$ of $R$ in $\mathcal{G}_M$ over $M$ as an ordered field, see \cite[Lemma 12]{KaplanskyMaximal1942}. By model completeness of  the theory of real closed fields, $\iota$ is an elementary embedding (in the language of ordered rings expanded by constants for the elements of $M$). Let $y \in \res(S \cap \mathcal{O}^n)$. Then there is $x \in S \cap \mathcal{O}^n$ such that $\res(x)=y$, so for all positive $\eps \in R$ we have $|x-\iota(y)|<\iota(\eps)$. By elementarity of $\iota$, for every positive $\eps \in R$ there is $z \in S(R)$ at distance $<\eps$ from $y$, so $y \in \overline{S(R)}$. For the other inclusion, if $y \in \overline{S(R)}$ then $\iota(y) \in \overline{S}$, so for $\eps \in \mathfrak{m}$ there is $x \in S \cap \mathcal{O}^n$ at distance $<\eps$ from $\iota(y)$. Then $\res(x)=\res(\iota(y))=y$, so $y \in \res(S \cap \mathcal{O}^n)$.
\end{proof}
    
    \added{For point (5) see also \cite[Lemma 2.1]{Marikova2010} where it is shown, for o-minimal expansions of fields, that the image of a definable set under the residue map of a convex valuation is closed.} Although we will not need it, we also have: 
\begin{rem} \mbox{} If $R$ is as in Lemma \ref{lem:germs}, then:
\begin{enumerate} 
    \item $R$ is weakly o-minimal;
    \item  Every $(\ell,n)$-$\texp$-polynomial that changes sign on $(-1,1)^\ell \times R^{n-\ell}$ has a zero.  
\end{enumerate}
Point (1) follows from \cite[Section 4]{Baisalov1998}.  The proof of (2) is similar to the proof of Lemma \ref{lem:germs}(4). 
\end{rem}

\section{Ax's theorem}\label{sec:Ax}

\begin{thm}[{\cite[Theorem 3]{Ax1971}}]\label{thm:ax}
 Let $F \supseteq \Q$ be a field and let $D:F\to F$ be a derivation. Let $x_1, \ldots, x_\ell, y_1, \ldots, y_\ell \in F^\times$ be such that $Dx_i = Dy_i/y_i$ for all $i=1, \ldots, \ell$. Assume that $x_1, \ldots, x_\ell$ are $\Q$-linearly independent over $\ker D$. Then $\td(x_1, \ldots, x_\ell, y_1, \ldots, y_\ell/\ker(D)) >\ell.$
\end{thm}

 By an application of the compactness theorem for first-order logic we obtain the following uniform version of Theorem \ref{thm:ax}, arguing as in \cite[Proposition 8]{Zilber2002}.

\begin{cor}\label{cor:uniform-ax}
    Let $R$ be a real closed field with a derivation $\delta:R \to R$. Let $Z \subseteq R^{2\ell}$ be a semialgebraic set defined over $M = \ker (\delta)$, with $\dim Z \leq \ell$. Then there is a finite set $\mathcal{U} \subseteq \mathbb{Z}^\ell\setminus\{0\}$ such that for every tuple $(x,y)=(x_1,\dots,x_\ell,y_1,\dots,y_\ell) \in Z$ with $y_i \neq 0$ and $\frac{\delta(y_i)}{y_i}=\delta(x_i)$  for all $i=1,\dots,\ell$ there is $u \in \mathcal U$ which satisfies $\delta(u \cdot x)=0$, where $u \cdot x$ denotes the scalar product.
\end{cor}

\begin{pf}
    Consider the set $\Sigma(x_1,\dots,x_\ell,y_1,\dots,y_\ell)$ of formulas in the language of ordered differential fields \[ \left\{ (x,y) \in Z\wedge \bigwedge_{i=1}^\ell y_i \neq 0 \wedge \bigwedge_{i=1
    }^\ell \frac{\delta(y_i)}{y_i}=\delta(x_i)\right\} \cup \left\{ \delta(u \cdot x) \neq 0 \mid u \in \mathbb{Z}^\ell \setminus \{0\} \right\}. \] If $\Sigma(x,y)$ is finitely satisfiable in $R$, then by the compactness theorem there are a real closed differential field $(R',\delta')$ extending $(R,\delta)$ and some $(\xi,\eta) \in Z(R')$ such that for all $i=1,\dots,\ell$ we have $\frac{\delta'(\eta_i)}{\eta_i}=\delta'(\xi_i)$ and $\delta'(u \cdot \xi) \neq 0$ for every $u \in \mathbb{Z}^\ell \setminus \{0\}$. On the other hand, by Ax's Theorem \ref{thm:ax} the inequality $\td(\xi,\eta/\ker(\delta')) \leq \dim Z \leq \ell$ implies that there is some $u \in \mathbb{Z}^\ell \setminus \{0\}$ such that $\delta'(u \cdot \xi)=0$, thus we have a contradiction. This shows that $\Sigma(x,y)$ is not finitely satisfiable, so there is a finite $\mathcal{U} \subseteq \mathbb{Z}^\ell \setminus \{0\}$ such that the formula \[ \left((x,y) \in Z \wedge \bigwedge_{i=1
    }^\ell \frac{\delta(y_i)}{y_i}=\delta(x_i)\right) \implies \bigvee_{u \in \mathcal{U}} \delta(u \cdot x)=0 \] holds for all tuples $(x,y)$ in the ordered differential field $R$.
\end{pf}

We will apply Corollary \ref{cor:uniform-ax} to the field $R$ of Definition \ref{defn:germs}, equipped with the derivation $\delta$ of Lemma \ref{lem:germs}.

\section{Khovanskii points are bounded}\label{sec:bounded}

The aim of the next two sections is to show the model completeness of $\Tres$. Fix models $M\subset N$ of $\Tres$. 

\begin{defn}\label{def:complexity}
\added{
    A quantifier-free formula $\varphi(x_1, \ldots, x_n)$ with parameters from $M$ has complexity $\leq \ell$ if it is equivalent in $\Tres$ to a disjunction of formulas $\bigvee_{i=1}^k \psi_i$, where each $\psi_i$ is, up to a permutation of the variables, of the form   $\exists \bar y \; \left(F_q  = 0 \land  \bigwedge_i z_i \in (-1,1)\right)$ where $\bar y$ is a tuple of variables, $F_q$ is an $\texp$-polynomial of complexity $\leq \ell$, and $\bar z \subset \bar x, \bar y$ are the variables that appear in the scope of $\texp$ in $F_q$.}
\end{defn}

\begin{lem}\label{rem:complexity}
\added{
Every quantifier-free formula $\varphi$ has complexity $\leq \ell$ for some $\ell$. 
}
\end{lem}

\begin{proof}\added{
     We may assume that the language includes the subtraction sign as we may write $u-v = c$ as $u = v+c$. Granted this, $\varphi$ is equivalent to a disjunction of conjunctions of formulas of the form $f=0$ or $g>0$ where $f,g$ are terms. We may eliminate the conjunctions using the equivalence $(f_1=0 \land f_2 = 0) \liff f_1^2+f_2^2 = 0$ and $f=0 \land g>0 \liff \exists x (f=0 \land x^2g - 1 = 0)$. 
     Using the equivalence 
      $p(\bar x, \texp(t)) = 0 \liff \exists y (y = t \land p(\bar x, \texp(y)) = 0)$ we may assume that each term is an $\texp$-polynomial, namely $\texp$ is only applied to variables. Our $\varphi$ is thus equivalent to a disjunction of formulas of the form $\exists \bar y F_q = 0$ where $F_q$ is an $\texp$-polynomial. Let $\bar z=(z_1,\dots,z_k) \subset \bar x, \bar y$ be the variables that appear in the scope of $\texp$ in $F_q$. Then $\exists \bar y F_q=0$ is equivalent to a disjunction of formulas of the form $\exists \bar y (F_q=0 \land \bigwedge_{i \in I} z_i \in (-1,1) \land \bigwedge_{j \in J} z_j \notin (-1,1))$,  where $I \cup J$ is a partition of $\{1,\dots,k\}$. We may now replace, for every $j \in J$, every occurrence of $\texp(z_j)$ in $F_q$ with $0$ and eliminate the inequalities expressing $z_j \nin (-1,1)$ using additional existential quantifiers, and eliminate the resulting conjunctions as we did above. 
      }
\end{proof}

\begin{exa}
The formula $x+y+\texp(x)=0$ is equivalent to 
\[\left(x\in (-1,1) \land x+y+\texp(x)=0 \right) \lor (x \nin (-1,1) \land x+y=0)\] and the second disjunct can be rewritten as
\[ \exists u((u^2(x+1)(x-1)-1)^2+(x+y)^2=0).\]
\end{exa}

\deleted{A quantifier-free formula $\varphi(x_1, \ldots, x_n)$ with parameters from $M$ has complexity $\leq \ell$ if, up to a permutation of the variables, it has the form $\left( \bigwedge_{i=1}^\ell -1 < x_i < 1 \right) \land  \psi(\bar x)$ where $\psi(\bar x)$ is a boolean combination of expressions of the form $F_p > 0$ or $F_q = 0$ where $F_p$ and $F_q$ are $\texp$-polynomials over $M$ of complexity $\leq \ell$.} 

\deleted{
\begin{rem}

Note that every quantifier-free formula without nested occurrences of $\texp$ is equivalent to a disjunction of formulas of complexity $\leq \ell$ for some $\ell$. 
For example $x+y>2 \land \texp(x)=z$ is equivalent to 
$$\left(x\in (-1,1) \land x+y>2 \land \texp(x)=z \right) \lor \left(x \nin (-1,1) \land x+y>2 \land 0 =z\right).$$
\end{rem}
}

\begin{defn}\label{A-B-conditions}
    We will denote by $A_{\ell}$ and $B_{\ell}$ the following statements.

    \begin{itemize}
        \item[$A_\ell$:]  Every Khovanskii point over $M$ of complexity $\leq \ell$ has coordinates in $M$.
        \item[$B_\ell$:]  Every Khovanskii point over $M$ of complexity $\leq \ell$ is $M$-bounded, namely it is bounded in norm by an element of $M$.
    \end{itemize}
For $\ell>0$, we will also write $A_{<\ell}$ and $B_{<\ell}$ for $A_{\ell-1}$ and $B_{\ell-1}$.
\end{defn}

\begin{prop}\label{prop:al-implies-complexity}
    Assume $A_\ell$. Then every quantifier-free formula over $M$ of complexity $\leq \ell$ is satisfiable in $M$ if and only if it is satisfiable in $N$. 
\end{prop}

\begin{proof}
    \deleted{Given a quantifier-free formula, we may eliminate the inequality symbols at the expense of introducing existential quantifiers and new variables, for example by writing $x>0$ as $\exists y (xy^2=1)$. It follows that any quantifier-free formula $\phi$ over $M$ of complexity $\leq \ell$ defines, up to a permutation of the variables, the projection of an $\texp$-variety over $M$ of the same complexity. Moreover, every $\texp$-variety over $M$ contains the projection of an $\texp$-Khovanskii point $P\in N^n$ over $M$ of the same complexity by Proposition \ref{prop:servi}. By $A_{\ell}$, the $\texp$-Khovanskii point has coordinates in $M$, and hence its projection satisfies $\phi$ in $M$.}

    \added{By Definition \ref{def:complexity}, it suffices to consider formulas of the form $F_q=0 \land \bigwedge_i z_i \in (-1,1)$, where $F_q$ is an $\texp$-polynomial over $M$ of complexity $\leq \ell$. Such a formula defines an $\texp$-variety $V$ over $M$ of complexity $\leq \ell$. By Proposition \ref{prop:servi}, the $\texp$-variety $V$ contains the projection of an $\texp$-Khovanskii point $P\in N^n$ over $M$ of the same complexity. By $A_{\ell}$, the point $P$ has coordinates in $M$.}
\end{proof}

To prove model completeness of $\Tres$ we will prove that $A_{\ell}$ holds for every $\ell \in \N$ by an inductive argument: in Proposition \ref{prop:M-bounded} we show that $A_{<\ell}$ implies $B_\ell$, and in Proposition \ref{prop:a<l-al} we use $B_{\ell}$ in the proof that $A_{<\ell}$ implies $A_\ell$. The model completeness result will then follow by Proposition \ref{prop:al-implies-complexity}.

\begin{lem}\label{lem:eliminate-dependences-pre}
Let $M \subseteq N$ be restricted exponential fields, and $(\alpha,\beta) \in N^n$ be an $(\ell,n)$-$\texp$-Khovanskii point over $M$. Suppose that $\alpha_\ell$ lies in the $\Q$-linear span of $\alpha_1, \ldots, \alpha_{\ell-1}$ modulo $M$. Then there is a non-zero integer $d\in \Z$ such that, letting $\alpha' = (\alpha_1/d, \ldots, \alpha_{\ell -1}/d)$,  we have that $(\alpha',\beta)$ is an $(\ell-1, n-1)$-Khovanskii point over $M$.
\end{lem}

\begin{proof}
    We will later prove a more general result, see Lemma \ref{lem:eliminate-dependences}.
\end{proof}

\begin{defn}
    Given $x \in N$, an element $y \in M$ is the \emph{standard part in $M$} of $x$ if $|x-y|$ is smaller than every positive element of $M$. The standard part in $M$ of an element $x$, if it exists, is necessarily unique, and we denote it by $\st(x)$. For $x=(x_1,\dots,x_n) \in N^n$ the standard part of $x$, if it exists, is $(\st(x_1),\dots,\st(x_n)) \in M^n$. 
\end{defn} 

\begin{lem}\label{lem:standard} Assume $A_{<\ell}$. Let $(\alpha,\beta) \in N^n$ be an $(\ell,n)$-Khovanskii point over $M$ that is not $M$-bounded. Then $\alpha$ has a standard part in $M$. 
\end{lem}
\begin{pf}
    By Lemma \ref{lem:exists-variety} we have that $(\alpha,\beta,\texp(\alpha))$ lies in a transversal intersection between $V$ and $G^{n,\ell}_{\exp}$ where $V$ is an algebraic variety over $M$ of dimension $\ell$, which is locally defined around $(\alpha,\beta,\texp(\alpha))$ by polynomial equations $p_1=\dots=p_n=0$ in $n+\ell$ variables. Let 
    $$
    W =V^{\reg}_{\ell,n+\ell}(p_1,\dots,p_n,y_1-\texp(x_1),\dots,y_{\ell-1}-\texp(x_{\ell-1}))
    $$ 
    be the set of regular points of $V \cap G_{\exp}^{n,\ell,-}$ in $U_{\ell,n+\ell}=(-1,1)^\ell \times N^n$. Then $W$ contains $(\alpha,\beta,\texp(\alpha))$ by Remark \ref{rem:meet-transversally}, it has dimension one by regularity, and it is defined by a formula of complexity $<\ell$ over $M$. Let $W_t = \{x\in W : |x| = t\}$. 
    By dimension considerations, $W_t$ is finite for all but finitely many $t$ (in both models $M$ and $N$). Work in the small model $M$. Let $\pi:M^{n+\ell}\to M^\ell$ be the projection to the first $\ell$ coordinates. Then there are $t_0 \in M$ and $k \in \N$ such that for all $t \in M$ with $t \geq t_0$ the set $D_t:= \pi(W_t) \subset M^\ell \cap (-1,1)^\ell$ has exactly $k$ elements. Moreover, $k >0$, as otherwise the quantifier-free formula $(x \in W \wedge |x|>t_0)$, which is satisfied in $N$ by $(\alpha,\beta,\texp(\alpha))$, would not be satisfied in $M$, contradicting $A_{\ell}$ and Proposition \ref{prop:al-implies-complexity}.  It follows that there are $M$-definable functions $f_1, \ldots, f_k$ ($k\neq 0$) such that, for $t \geq t_0$ in $M$, $D_t = \{f_1(t), \ldots, f_k(t)\} \subseteq (-1,1)^\ell$.  By o-minimality the limits $a_i = \lim_{t\to +\infty} f_i(t)$ exist in $M^\ell \cap [-1,1]^\ell$. Let $D_\infty = \{a_1, \ldots, a_k\} \subseteq [-1,1]^\ell$. Then for every $\eps \in M$ there is $r \in M$ such that the implication \[ \left(x \in W \wedge |x|>r\right) \implies \bigvee_{i=1}^k |\pi(x) - a_i| \leq \eps   \] holds for every $x\in M^{n+\ell}$. By $A_{<\ell}$ and Proposition \ref{prop:al-implies-complexity} the implication holds for all $x \in N^{n+\ell}$, since it is expressible by a quantifier-free formula of complexity $< \ell$. 
   Since $(\alpha, \beta)$ is not $M$-bounded, we have $|(\alpha,\beta,\texp(\alpha))|>r $ for every $r \in M$. Since $\pi(\alpha, \beta, \texp(\alpha)) = \alpha$, by the above implication $\bigvee_{i=1}^k |\alpha - a_i| \leq \eps$ holds for all positive $\eps \in M$. It follows that $\alpha$ has a standard part in $\{a_1, \ldots, a_k\} \subset M$.  
\end{pf}

\begin{prop}\label{prop:M-bounded} 
 $A_{<\ell} \implies B_\ell$. 
\end{prop}

\begin{proof} Assume $A_{<\ell}$ and let $P=(\alpha,\beta) \in N^n$ be a $(\ell,n)$-$\texp$-Khovanskii point over $M$. We need to prove that $P$ is $M$-bounded.

    If $\ell = 0$, $P$ is algebraic over $M$, so it belongs to $M$. So assume $\ell > 0$. By Lemma \ref{lem:eliminate-dependences-pre} we may assume $\ldim_{\mathbb{Q}}(\alpha/M)=\ell$. By Lemma \ref{lem:exists-variety}, $(\alpha,\beta,\texp(\alpha))$ lies in a transversal intersection between $G_{\exp}^{n,\ell}$ and an algebraic variety $V \subseteq N^{n+\ell}$ of dimension $\ell$. Assume for a contradiction that $(\alpha,\beta)$ is not $M$-bounded. Then, by Lemma \ref{lem:standard},  $\alpha$ has a standard part $a = \st(\alpha)$ in $M^\ell$. Moreover, since $\alpha \in (-1,1)^\ell$, we have $a=(a_1,\ldots,a_\ell)\in [-1,1]^\ell$. 
    
    Let $\pi$ denote the projection map that sends $(x_1,\dots,x_\ell,y_1,\dots,y_{n-\ell},z_1,\dots,z_{\ell})$ to $(x_1,\dots,x_{\ell},z_1,\dots,z_{\ell})$. So $\pi(\alpha, \beta, \texp(\alpha)) = (\alpha, \texp(\alpha))$. Let $Z:=\overline{\pi(V)}$ be the topological closure of $\pi(V)$. We can define $Z$ by a semialgebraic formula with parameters from $M$:
    $$(x,z) \in Z :\iff \forall \eps>0 \;\exists (x',y,z') \in V :  |(x,z) - (x',z')| < \eps$$
    where $x, z, x', z'$ are $\ell$-tuples and $y$ is a tuple of length $n-\ell$. 
    
    Let $\mathcal G_M$, $\mathcal O \subseteq \mathcal G_M$ and $R = \mathcal O/\mathfrak m$ be as in Definition \ref{defn:germs}. Since $\mathcal G_M$ and $R$ are real closed fields containing $M$, and $Z$ is a semialgebraic set defined over $M$, we can interpret $Z$ both in $\mathcal G_M$ and $R$. In any real closed field (and more generally in any o-minimal structure) the topological closure of a semialgebraic set has the same dimension as the set \cite[Chapter IV, (1.18)]{vdDries1998}. Moreover the dimension cannot increase under a projection, so $\dim Z \leq \dim V = \ell$.

    Consider the derivation $\delta:R\to R$ given by Lemma \ref{lem:germs}. Since $\dim Z \leq \ell$, 
    by Corollary \ref{cor:uniform-ax} there is a finite set $\mathcal{U} \subseteq \mathbb{Z}^d \setminus \{0\}$ such that, in $R$, for every $(x,z) = (x_1, \ldots, x_\ell, z_1, \ldots, z_\ell) \in Z$ with $\bigwedge_{i=1}^\ell \delta(z_i)/z_i=\delta(x_i)$, there is $u \in \mathcal{U}$ such that $\delta(u \cdot x)=0$, where $u\cdot x = \sum_{i=1}^\ell u_i x_i$ is the scalar product. 

Now recall that $a = \st(\alpha)\in M^\ell$. Since $\mathcal U$ is finite, there is $d = (d_1, \ldots, d_\ell) \in \mathcal U$ such that for all $u\in \mathcal U$ we have $|d \cdot \alpha - d \cdot a| \leq |u \cdot \alpha - u \cdot a|$. Since $\ldim_{\mathbb{Q}}(\alpha/M)=\ell$, we have $|d \cdot \alpha - d \cdot a| \neq 0$. By permuting the coordinates of $\alpha$, we may assume that $d_\ell \neq 0$. 

Since $a \in [-1,1]^\ell$ and $d\in \Z^\ell$, there is some $n_0\in \N$ such that $|d\cdot a| < n_0$ and $|d_i a_i| <n_0$ for $i=1, \ldots, \ell$. We can then define $e^x$ as $\texp(x/n_0)^{n_0}$, so that the function $x \mapsto e^x$ is smooth on $(-n_0,n_0)$ and satisfies $e^{x+y}=e^xe^y$ whenever $x,y,x+y$ are all in $(-n_0,n_0)$ (Propositions \ref{prop:exp-finite} and \ref{prop:exp-finite-2}).

Let \[W \subseteq V \cap G_{\exp}^{n,\ell,-}\] denote the set of points of transversal intersection between $G_{\exp}^{n,\ell,-}$ and $V$, which is a $1$-dimensional definable smooth manifold containing $(\alpha,\beta,\texp(\alpha))$. 
    For every $t \in M^{>0}$, consider the quantifier-free formula $\chi_t(x,y,z)$ given by the conjunction of the following:
    \begin{itemize}
        \item[(i)] $(x,y,z) \in W$,
        \item[(ii)] $|x-a| \leq t^{-1} \wedge |z-e^a|\leq t^{-1}$,
        \item[(iii)] $\bigwedge_{u \in \mathcal{U}} 0<|d\cdot x - d\cdot a| \leq |u \cdot x - u \cdot a|$,
    \end{itemize}
    where $x,z$ are $\ell$-tuples and $y$ is a tuple of length $n-\ell$. 
    
    The Khovanskii point $(\alpha,\beta,e^\alpha) \in N^{n+\ell}$ satisfies $\chi_t$ for every $t \in M^{>0}$ (for point (iii) use $\ldim_\Q(\alpha/M) = \ell > 0$). 
    Since $\chi_t(x,y,z)$ has complexity $<\ell$ and we are assuming $A_{<\ell}$, it follows that the formula $\chi_t(x,y,z)$ is satisfiable in $M$ for every $t \in M^{>0}$. Since $\chi_{t'}(M) \subseteq \chi_{t}(M)$ for $t'>t$, in each $\chi_t(M)$ there are points $(x,y,z)$ with $x$ arbitrarily close to $a$ (by point (ii)), but different from $a$ (by point (iii)). It follows that, given $t_0\in M^{>0}$, there is $t_1 \in M^{>0}$ such that, for each $t \in M$ with $t\geq t_1$, there is $(x,y,z) \in \chi_{t_0}(M)$ such that $|d\cdot x - d\cdot a| = t^{-1}$. 
    
    The sets 
   \begin{equation}\label{e:Ltdef}
        L_t = \{(x,y,z) \in M^{n+\ell} : |d \cdot x - d\cdot a| = t^{-1}\}
   \end{equation}
    are all disjoint as $t$ varies, so for all $t$ outside a finite set the intersection $L_t \cap \chi_{t_0}(M)$ is finite (because by point (i) $\chi_{t_0}(M) \subseteq W$, which has dimension one). By o-minimality we can assume (increasing $t_1$ if necessary), that there is $k\in \N$ such that 
    \begin{equation}\label{e:Lcapchi}
    L_t \cap \chi_{t_0}(M) = \{f_1(t), \ldots, f_k(t)\}
    \end{equation}
    for each $t\geq t_1$, where $k\neq 0$ and $f_1, \ldots, f_k$ are continuous $M$-definable functions from $[t_1,+\infty)$ to $M^{n+\ell}$. Consider one of the functions 
    $$f\in \{f_1, \ldots, f_k\}.$$
    Since $f(t) \in \chi_{t_0}(M) \subseteq G^{n,\ell,-}_{\exp}$, we can write the $n+\ell$ components of $f(t)$ in the form  
    \begin{align*}f(t) &=  (\phi_1(t), \ldots, \phi_\ell(t), \theta_1(t), \ldots, \theta_{n-\ell}(t), \psi_1(t), \ldots, \psi_\ell(t)) \\
    & = (\Phi(t), \Theta(t), \Psi(t))
    \end{align*}
where $\phi_i(t) \in (-1,1)$ for $i=1, \ldots, \ell$ and $\psi_i(t) = e^{\phi_i(t)}$ for $i = 1, \ldots, \ell-1$. 

    By construction, 
    \begin{equation}\label{e:Phi}
    (\Phi(t), \Theta(t), \Psi(t)) \in L_t \cap \chi_{t_0}(M)
    \end{equation}
    for every sufficiently large $t\in M$. 
    By (i), for each $t\in M$ with $t\geq t_1$ we have $(\Phi(t), \Theta(t), \Psi(t)) \in V$, and hence \[(\germ(\Phi),\germ(\Theta),\germ(\Psi)) \in V(\mathcal G_M)\] so we have 
$$(\germ(\Phi), \germ(\Psi)) \in \pi(V)(\mathcal G_M).$$

 The components of $\Phi$ and $\Psi$ are bounded by (ii) (hence polynomially bounded), so $(\germ(\Phi),\germ(\Psi)) \in \mathcal{O}^{2\ell}$. 
 We may then take their residues $[\Phi] = \germ(\Phi)+\mathfrak{m}$ and $[\Psi] = \germ(\Psi) + \mathfrak{m}$ in $R = \mathcal O/\mathfrak{m}$.
By Lemma \ref{lem:germs}(\ref{item:semi-closure}), since $Z=\overline{\pi(V)}$, it follows that 
$$([\Phi],[\Psi]) \in Z(R).$$
Recall that $e^{\phi_i(t)} = \psi_i(t)$ for $i=1, \ldots, \ell-1$. 
   \begin{claim}\label{claim:diferent-residues}
        We have $[e^{\phi_\ell}] \neq [\psi_\ell]$.
    \end{claim} 
    \begin{pf}
    For a contradiction, assume $[e^{\phi_\ell}]=[\psi_\ell]$. Then $([\Phi],[e^{\Phi}]) = ([\Phi],[\Psi]) \in Z(R)$ and $\delta([\psi_i])/ [\psi_i]=\delta([\phi_i])$ for $i=1,\dots,\ell$ by Lemma \ref{lem:germs}(\ref{diff-eq}). By definition of $\mathcal{U}$, there is $u \in \mathcal{U}$ such that $\delta(u \cdot [\Phi])=0$, so there is a constant $c\in \ker(\delta) = M$ such that $u\cdot [\Phi] = c$. 
   This means that $\germ(u\cdot \Phi - c)$ is in the maximal ideal $\mathfrak m$, and in particular $u \cdot \Phi(t) - c$ tends to $0$ when $t\to +\infty$. 
    By (ii) $ \Phi(t) \to  a$, so we have that $c=u \cdot a$. By (iii) it follows that 
    \begin{align*}
        0&=|u \cdot [\Phi]-u \cdot a| \\
    &\geq |d \cdot [\Phi]-d \cdot a|, 
    \end{align*}
    and therefore $\germ(d \cdot \Phi - d \cdot a)$ is in the maximal ideal $\mathfrak{m}$. By (\ref{e:Phi}), $f(t) = (\Phi(t), \Theta(t), \Psi(t)) \in L_t$, so $|d \cdot \Phi(t)-d \cdot a|=t^{-1}$ by (\ref{e:Ltdef}), which is absurd since the germ of $t \mapsto t^{-1}$ is not in the maximal ideal.
    This proves the claim. 
    \end{pf}
    
        Let $r = \lim_{t\to +\infty} \prod_{i=1}^{\ell-1} \psi_i(t)^{d_i} \in M$.

        Then $r>0$ and for all sufficiently large $t\in M$ 
    \begin{align*}
        |e^{d \cdot \Phi(t)}-\Psi(t)^{d}| &= \left| e^{d_{1}\phi_1(t)}\cdots e^{d_{\ell}\phi_\ell(t)}-\psi_1(t)^{d_{1}} \cdots \psi_\ell(t)^{d_{\ell}} \right|\\
        &=\psi_1(t)^{d_1} \cdots \psi_{\ell-1}(t)^{d_{\ell-1}}
        \left|e^{d_{\ell}\phi_\ell(t)}-\psi_\ell(t)^{d_{\ell}}\right| \\
        & \geq \tfrac r 2 |e^{d_{\ell}\phi_\ell(t)}-\psi_\ell(t)^{d_{\ell}}|.
    \end{align*}
By Claim \ref{claim:diferent-residues} $[e^{\phi_\ell}] \neq [\psi_\ell]$ and since by assumption $d_\ell \neq 0$, we have $[e^{\phi_\ell}]^{d_\ell} \neq [\psi_\ell]^{d_\ell}$, namely the germ of $e^{d_\ell \phi_\ell} - \psi_\ell^{d_\ell}$ is not in the maximal ideal $\mathfrak m$. This implies that there is some $m\in \N$ such that for all sufficiently large $t\in M$ we have 
\[
\tfrac r 2 |e^{d_{\ell}\phi_\ell(t)}-\psi_\ell(t)^{d_{\ell}}| \geq t^{-m} = |d \cdot \Phi(t) - d \cdot a|^{m}.
\]
where the equality follows from (\ref{e:Ltdef}). 
Combining this with the previous inequalities we obtain  
\begin{equation}\label{e:psi-ineq}
    |e^{d \cdot \Phi(t)}-\Psi(t)^{d}| \geq |d \cdot \Phi(t) - d \cdot a|^{m}.
\end{equation}

    By the Taylor formula for $e^t=\texp(t/n_0)^{n_0}$ centered at $d \cdot a\in M \cap (-n_0,n_0)$ there is a polynomial $h$ over $M$ and some $\eps\in M^{>0}$ such that
    \begin{equation}\label{e:Taylor}
        |h(t)-e^t| \leq |t-d \cdot a|^{m+1} 
    \end{equation}
    whenever $|t- d\cdot a| < \eps$. This inequality is also valid for $t$ in the big model $N \supseteq M$. Since $\eps \in M$ and $a = \st(\alpha)$, the inequality holds for $t = d\cdot \alpha$, that is
\begin{equation}\label{e:khov-ineq}
    |h(d\cdot \alpha)-e^{d\cdot \alpha}| \leq |d\cdot \alpha - d \cdot a|^{m+1}.
\end{equation}
Recall that $m$, hence $h$, depend on the choice of $f \in \{f_1,\dots,f_k\}$. Define \[S_f:=\{ (x,y,z) : |h(d \cdot x) - z^{d}|>|d \cdot x - d \cdot a|^{m+1} \} \]
where $z^d = \prod_{i=1}^{\ell}z_i^{d_i}$. 
By (\ref{e:khov-ineq}) we have \[(\alpha,\beta,e^{\alpha}) \notin S_f.\] On the other hand
    for any sufficiently large $t\in M$ we have $0<|d\cdot \Phi(t) - d\cdot a| < \eps$ and, by (\ref{e:psi-ineq}) and (\ref{e:Taylor}),
    \begin{align*}
        |h(d \cdot \Phi(t)) - \Psi(t)^{d}| & = |h(d \cdot \Phi(t)) - e^{d \cdot \Phi(t)} +e^{d \cdot \Phi(t)} -\Psi(t)^{d}|\\
        & \geq |e^{d \cdot \Phi(t)} -\Psi(t)^{d}| - |h(d \cdot \Phi(t)) - e^{d \cdot \Phi(t)}| \\
        & \geq |d \cdot \Phi(t)- d \cdot a|^{m} - |d \cdot \Phi(t)- d \cdot a|^{m+1} \\
        &>|d \cdot \Phi(t)- d \cdot a|^{m+1}.
    \end{align*}
    So for every sufficiently large $t$ we have: \begin{equation}\label{e:f-in-sf}
        f(t)=(\Phi(t),\Theta(t),\Psi(t)) \in S_f.
    \end{equation}

    Now consider the union $S=\bigcup_{f} S_f$ where $f$ ranges in $\{f_1, \ldots, f_k\}$. Then $S$ does not contain $(\alpha,\beta,e^\alpha)$ and is semialgebraic over $M$. 
    
    \begin{claim} For all $t\in M$ sufficiently large, $\chi_t(M) \subseteq S$. 
    \end{claim}
    \begin{proof}
    By (\ref{e:Lcapchi}) and (\ref{e:f-in-sf}), $L_t \cap \chi_{t_0}(M) = \{f_1(t), \ldots, f_k(t)\} \subseteq S$ for all $t\in M$ sufficiently large, say $t\geq t_2$. Hence it suffices to show that, for all sufficiently large $t'$, we have $\chi_{t'}(M) \subseteq \bigcup_{t \geq t_2} L_t \cap \chi_{t_0}(M)$. To this end, we observe that if $t'$ is sufficiently large then $\forall x(0<|x-a| \leq t'^{-1} \implies |d\cdot x - d \cdot a|^{-1} \geq t_2)$ holds, hence if $x \in \chi_{t'}(M)$ then $x \in L_{t} \cap \chi_{t_0}(M)$ with $t=|d \cdot x -d \cdot a|^{-1} \geq t_2$.
    \end{proof}
    
    We have thus proved that for all sufficiently large $t$ the quantifier-free formula $\chi_t(x,y,z) \wedge (x,y,z) \notin S$ is not satisfiable in $M$. However it is satisfied in $N$ by $(\alpha,\beta,e^\alpha)$, and since it has complexity $<\ell$ we have a contradiction with the hypothesis $A_{<\ell}$ and Proposition \ref{prop:al-implies-complexity}.
\end{proof}

\begin{cor}\label{cor:no-st-part}
    Assume $A_{<\ell}$. Let $Q=(q_1,\dots,q_n)$ be an $(\ell,n)$-$\texp$-Khovanskii point over $M$. If one of the coordinates $q_i$ has a standard part in $M$, then $q_i \in M$. 
\end{cor}

\begin{proof}
    Assume $Q$ solves a Khovanskii system of $(\ell,n)$-$\texp$-polynomials $g_1,\dots,g_n$ and that some coordinate $q_i$ has a standard part $\st(q_i) \neq q_i$ in $M$. Consider the polynomial $g_{n+1}:=x_{n+1}(x_i-\st(q_i))-1$. Then $(q_1,\dots,q_n,1/(q_i-\st(q_i))$ solves the Khovanskii system given by the $(\ell,n+1)$-$\texp$-polynomials $g_1,\dots,g_{n+1}$, so it is a Khovanskii point of complexity $\leq \ell$ which is not $M$-bounded. This contradicts the fact that $A_{<\ell}$ implies $B_\ell$ (Proposition \ref{prop:M-bounded}).
\end{proof}

\section{Model completeness} \label{sec:model-complete}

The next two propositions are a reformulation of \cite[Lemma 2.8 and Lemma 6.3]{Wilkie1996}, although the hypotheses are different. We include proofs to keep the paper self-contained. As in the previous section, we fix models $M\subset N$ of $\Tres$. 

\begin{prop}\label{prop:enumerate-multifunctions} Assume $A_{<\ell}$. 
    Let $Q \in N^n$ be an $(\ell,n)$-Khovanskii point given by a system of $(\ell,n)$-$\texp$-polynomials $g_1, \ldots, g_{n}$ over $M$.  Suppose that: 
    \begin{itemize}
        \item[(i)] $g_1, \ldots, g_{n-1}$ have complexity $\leq \ell-1$. 
        \item[(ii)]  There is $i \in \{1,\dots,n\}$ such that   $\det\left(\frac{\partial(g_1, \ldots, g_{n-1})}{\partial(x_1, \ldots,x_{i-1},x_{i+1}, \dots, x_n)}\right)(x) \neq 0$ for all $x\in V_{\ell,n}(g_1, \ldots, g_{n-1})$.
        \item[(iii)] $\det\left(\frac{\partial(g_1, \ldots, g_{n})}{\partial(x_1, \ldots, x_n)}\right)(x) > 0$ for all $x \in V_{\ell,n}(g_1,\dots,g_n)$.
    \end{itemize}
     Then $Q \in M^n$.
\end{prop}

\begin{proof}  Let $V:=V_{\ell,n}(g_1,\dots,g_{n-1})$. By assumption (ii), every point of $V$ is regular, so $V$ has dimension $1$. For simplicity we assume the index $i$ in (ii) is $1$, the other cases being similar. Thus (ii) becomes $\det\left(\frac{\partial(g_1, \ldots, g_{n-1})}{\partial(x_2, \dots, x_n)}\right)(x) \neq 0$ for all $x\in V$. This implies that 
the projection $\pi:V\to N$ on the first coordinate is a local diffeomorphism and therefore it has finite fibers by o-minimality.

     By $A_{<\ell}$ and Proposition \ref{prop:M-bounded}, every Khovanskii point over $M$ of complexity $\leq \ell$ is $M$-bounded, so the coordinates $(q_1,\dots,q_n)$ of $Q$ are $M$-bounded.

\begin{claim} \label{claim:box}
    There are $k \in \N$, $a<b$ in $M$ and a positive $c \in M$ such that, letting $X = [a,b]$ and $Y = \{y\in N^{n-1}: |y|\leq c\}$, the set $F = (X\times Y) \cap V$ contains $Q$ and it can be written as a disjoint union $\bigcup_{i=1}^k f_i$ where each $f_i$ is an $N$-definable continuous differentiable function from $X$ to $Y$. 
    
    Similarly, $F \cap M^n$ can be written as a disjoint union $\bigcup_{i=1}^k f'_i$ where each $f'_i$ is an $M$-definable continuous differentiable function from $X$ to $Y$ (interpreted in $M$). 
\end{claim}
Note that we are not claiming that $f_i'$ is the trace of $f_i$ in the small model $M$, and in fact a priori the trace may not be definable.  

\begin{proof}[Proof of claim.] Work in the big model $N$. Since $V$ has dimension $1$, for all but finitely many $r \in N$ the set \[V(r) = \{(x,y) \in V: x\in N, y \in N^{n-1}, |y|=r \}\] is finite.
Let $\xi_1, \ldots, \xi_k\in N^{n-1}$ be pairwise distinct points such that  the $M$-bounded points of $\pi^{-1}(q_1)$ are exactly $(q_1,\xi_1),\dots,(q_1,\xi_k)$.

Let $c\in M$ be strictly larger than $\max_{i=1}^k |\xi_i|$ and such that $V(c)$ is finite. Since $V(c)$ is defined by a formula of complexity $<\ell$ with parameters in $M$, it follows by $A_{<\ell}$ and Proposition \ref{prop:al-implies-complexity} that $V(c) \subseteq M^n$. 

Consider the projection $D = \pi(V(c))$ on the first coordinate. Then $D$ is a finite subset of $(-1,1)\cap M$, and it does not contain $q_1$ by the choice of $c$. Let $I$ be the maximal open sub-interval of $(-1,1)$ containing $q_1$ and disjoint from $D$. If $q_1$ has standard part in $M$ equal to one of the endpoints of $I$, then by Corollary \ref{cor:no-st-part} it coincides with that endpoint, which is absurd since $I$ is open and $q_1 \in I$. It follows that there is a closed sub-interval 
$$X=[a,b] \subseteq I,$$ 
with endpoints in $M$, such that $q_1 \in X$. 
 Define 
$$Y = \{y\in N^{n-1}: |y|\leq c\}, \qquad F = (X\times Y) \cap V.$$ Since $X \cap D = \emptyset$, there are no points $(x,y)\in F$ with $|y| = c$.

For $m\in \N$, let $X_{m} \subseteq X$ be the set of all $x\in X$ such that the fiber $\pi^{-1}(x) \cap F$ has cardinality $ \geq m$. 
Note that $\pi: F \to X$ is a local diffeomorphism by (ii), so $X_{m}$ is open in $X$ for all $m$. 

We will show that $X_{m}$ is closed. To this end consider a point $x_0\in X$ in the closure of $X_{m}$. Then there are an open interval $U\subset X_m$ having $x_0$ in its boundary and definable continuous functions $h_1, \ldots, h_m$ from $U$ to $Y$ such that for every $x\in U$, the fiber $\pi^{-1}(x) \cap F$ contains $\{(x,h_1(x)), \ldots, (x,h_m(x))\}$ and these $m$ points are distinct. Since $Y$ is closed and bounded, the limits $a_i = \lim_{x\to x_0, x\in U} h_i(x) \in N^{n-1}$ exist and are in $Y$. Moreover these limits must be distinct because if, say, $a_i = a_j$ for some $j\neq i$, then $\pi:F \to X$ would not be a local diffeomorphism at $(x_0,a_i)$. 

We have thus proved that each $X_m$ is clopen in $X$, so it is either $X$ or empty. Since $q_1 \in X_k \setminus X_{k+1}$, all the fibers $\pi^{-1}(x) \cap F$ have cardinality $k$. It follows that there are $N$-definable functions $h_1,\dots,h_k:X \to Y$ such that for every $x \in X$ we have $\pi^{-1}(x) \cap F=\{(x,h_1(x)),\dots,(x,h_k(x))\}$. These functions can be chosen to be continuous on each cell of a definable cell decomposition of the interval $X =[a,b]$. By a gluing argument, it follows that there are $N$-definable continuous functions $f_1,\dots,f_k:X \to Y$ such that each $f_i$ coincides on any given cell of the decomposition with one of the $h_j$ and $\bigcup_{i=1}^k h_i = \bigcup_{i=1}^k f_i = F$. 

Granted the continuity, $f_1, \ldots, f_k$ are also differentiable by the o-minimal implicit function theorem \cite[Chapter V, (2.11)]{vdDries1998} since they are sections of the local diffeomorphism $\pi:F\to X$.

To prove the second part, we first observe that for $x \in M$ every fiber $\pi^{-1}(x) \cap F$ is a finite set defined by a quantifier-free formula over $M$ of complexity $<\ell$, so by $A_{<\ell}$ and Proposition \ref{prop:al-implies-complexity} we have $\pi^{-1}(x) \cap F \subseteq M^n$. It follows that all the fibers of $F \cap M^n$ also have cardinality $k$. Repeating the argument in the previous paragraph working in $M$, we obtain the $M$-definable functions $f_1',\dots,f_k'$.
\end{proof}
\begin{claim}
    $(X\times Y) \cap V_{\ell,n}(g_1, \ldots, g_n) \subset M^n$.
\end{claim}

\begin{proof} Recall that $X = [a,b]$ with $a,b\in M$. Since $q_1 \in (a,b)$, by Corollary \ref{cor:no-st-part} it does not have standard part in $M$ equal to $a$ or $b$. Hence we may shrink the interval if necessary, preserving the conclusions of Claim \ref{claim:box}, to assume that $g_n$ has no zeros on $F$ with first coordinate $a$ or $b$. Let \[\{\alpha_1,\dots,\alpha_k\} =\{x \in Y \mid (a,x) \in F\}\] \[\{\beta_1,\dots,\beta_k\}=\{x \in Y \mid (b,x) \in F\}.\] Since $a,b\in M$, by $A_{<\ell}$ and Proposition \ref{prop:al-implies-complexity}, $\alpha_1,\dots,\alpha_k,\beta_1,\dots,\beta_k$ lie in $M^{n-1}$. Let $J(x_1,\dots,x_n)$ be the $\texp$-polynomial $(-1)^{n+1}\det \left( \frac{\partial (g_1,\dots,g_{n-1})}{\partial(x_2,\dots,x_n)}\right)(x_1,\dots,x_n)$. Note that $J$ has complexity $\leq \ell-1$ by (i).  

    We claim that the number of zeros of $g_n$ on $F$  is given, both in $M$ and $N$, by the quantity $C_1 - D_1+ C_2-D_2$ where:
    \begin{itemize}
        \item $C_1$ is the number of indices $i$ such that $g_n(a,\alpha_i) <0$ and $J(a,\alpha_i)>0$. 
        \item $C_2$ is the number of indices $i$ such that $g_n(a,\alpha_i)>0$ and $J(a,\alpha_i)<0$.
        \item $D_1$ is the number of indices $i$ such that $g_n(b,\beta_i)<0$ and $J(b,\beta_i)>0$. 
        \item $D_2$ is the number of indices $i$ such that $g_n(b,\beta_i)>0$ and $J(b,\beta_i)<0$. 
    \end{itemize}
Consider first the model $N$. 
   Let $\{f_1,\dots,f_k\}$ be the $N$-definable functions obtained in Claim \ref{claim:box}. Denote by $\sigma$ the permutation on $\{1,\dots,k\}$ defined by $\sigma(i)=j$ if there is $\phi \in \{f_1,\dots,f_k\}$ such that $\alpha_i=\phi(a)$ and $\beta_j=\phi(b)$. Now fix $\phi \in\{f_1,\dots,f_k\}$. Consider the function sending $t \in [a,b]$ to $ J(t,\phi(t))$. By (ii), $J$ has no zeros on $V$, so $J(t,\phi(t))$ has the same sign for every $t \in [a,b]$.
    
    Write $\phi(t)=(\phi_2(t),\dots,\phi_n(t)).$ Then for every $i \in \{1,\dots,n\}$ we have \begin{align*}
        \frac{d}{dt} (g_i(t,\phi(t))) &=\frac{\partial}{\partial x_1}g_i (t,\phi(t)) + \sum_{j=2}^n \frac{\partial}{\partial x_j}g_i(t,\phi(t)) \frac{d}{dt} \phi_{j}(t).
    \end{align*}

    For $i=1,\dots,n-1$ we have that $g_i(t,\phi(t))$ is identically $0$ since every point of the form $(t,\phi(t))$ lies in $V=V_{\ell,n}(g_1,\dots,g_{n-1})$, so the expression above vanishes. Hence multiplying the matrix $\left(\frac{\partial (g_1,\dots,g_n)}{\partial(x_1,\dots,x_n)} \right) (t, \phi(t))$ by the vector $\left(1,\frac{d}{dt} \phi_{2}(t), \dots, \frac{d}{dt} \phi_{n}(t)\right)$ we obtain $\left(0,\dots,0, \frac{d}{dt}g_n(t,\phi(t)) \right).$ By Cramer's rule it follows that, for every $t$ such that $\det\left(\frac{\partial (g_1, \ldots, g_{n})}{\partial (x_1, \ldots, x_n)}\right)(t,\phi(t)) \neq 0$, the equality
   \begin{align*}
        1 &= (-1)^{n+1} \frac{d}{dt}g_n(t,\phi(t)) \left|\frac{\partial (g_1, \ldots, g_{n-1})}{\partial (x_2, \ldots, x_n)}(t,\phi(t))\right|\left|\frac{\partial (g_1, \ldots, g_{n})}{\partial (x_1, \ldots, x_n)}(t,\phi(t))\right|^{-1}
    \end{align*}
    holds, where $|A|= \det(A)$. 

    Hence by (iii), if $(t,\phi(t)) \in V_{\ell,n}(g_1,\dots,g_n)$ then $\frac{d}{dt}g_n(t,\phi(t))$ is non-zero and has the same sign as $J(t,\phi(t))$. This implies that $g_n(t,\phi(t))$ has at most one zero on $[a,b]$. More precisely, $g_n(t,\phi(t))$ has a zero on $[a,b]$ if and only if one of the following holds:
    \begin{itemize}
        \item $g_n(a,\alpha_i)<0$ and $g_n(b,\beta_{\sigma(i)})>0$, or
        \item $g_n(a,\alpha_i) >0$ and $g_n(b,\beta_{\sigma(i)})<0$.
    \end{itemize}
    Note that in the first case $J(a,\alpha_i)$ and $J(b,\beta_{\sigma(i)})$ are positive, giving a positive contribution to $C_1-D_1$, while in the second case $J(a,\alpha_i)$ and $J(b,\beta_{\sigma(i)})$ are negative, giving a positive contribution to $C_2-D_2$.
    
    Repeating the argument as $\phi$ varies in $\{f_1,\dots,f_k\}$, we see that the number of indices $i$ for which the first item holds is equal to $C_1-D_1$, while the number of indices for which the second item holds is equal to $C_2-D_2$. This shows that the number of zeros of $g_n$ in $(X \times Y) \cap V$ in the model $N$ is given by $C_1-D_1+C_2-D_2$. 
    
    To conclude the proof of the claim, we repeat the argument working in $M$. Now $\phi$ varies among the $M$-definable functions $\{f_1',\dots,f_k'\}$, and we obtain a different permutation $\sigma'$ on $\{1,\dots,k\}$. However it follows from the definitions that $C_1,C_2,D_1,D_2$ do not depend on the permutation, hence the number of zeros of $g_n$ in $(X \times Y) \cap V$ is the same in the two models.
    \end{proof}
    To conclude the proof it suffices to recall that $Q \in (X \times Y) \cap V_{\ell,n}(g_1, \ldots, g_n)$.
\end{proof}

\begin{prop}\label{prop:a<l-al}
    $A_{<\ell} \implies A_\ell$.
\end{prop}

\begin{proof}
    By Proposition \ref{prop:enumerate-multifunctions}, it suffices to show that for every $(\ell,n)$-$\texp$-Khovanskii point $Q=(q_1,\dots,q_n) \in  N^n$ there are $(\ell,m)$-$\texp$-polynomials $g_1,\dots,g_{m}$ and a $\texp$-Khovanskii point $Q' \in V_{\ell,m}^{\reg}(g_1,\dots,g_{m})$, for some $m\geq n$, such that $Q$ is a projection of $Q'$ and the following hold: 
    \begin{itemize}
        \item[(i)] $g_1, \ldots, g_{m-1}$ have complexity $\leq \ell-1$. 
        \item[(ii)] There is $i \in {1,\dots,m}$ such that    $\det\left(\frac{\partial(g_1, \ldots, g_{m-1})}{\partial(x_1, \ldots, x_{i-1},x_{i+1},\dots x_{m})}\right)(x) \neq 0$ for all $x\in V_{\ell,m}(g_1, \ldots, g_{m-1})$.
        \item[(iii)] $\det\left(\frac{\partial(g_1, \ldots, g_{m})}{\partial(x_1, \ldots, x_{m})}\right)(x) > 0$ for all $x \in V_{\ell,m}(g_1,\dots,g_{m})$.
    \end{itemize}

    So assume that $Q \in V_{\ell,n}^{\reg}(h_1,\dots,h_n)$ for $(\ell,n)$-$\texp$-polynomials $h_1,\dots,h_n$. Replacing $h_1$ by $-h_1$ if necessary, we may assume that $\det(\frac{\partial (h_1, \ldots, h_{n})}{\partial (x_1, \ldots, x_n)})(Q) > 0$. Let $h_{n+1}$ be the $(\ell,n+1)$-$\texp$-polynomial 
    \[x_{n+1}^2 \det\left(\frac{\partial (h_1, \ldots, h_{n})}{\partial (x_1, \ldots, x_n)}\right)(x_1,\dots,x_n) -1\]
    and let $q_{n+1}=\det(\frac{\partial (h_1, \ldots, h_{n})}{\partial (x_1, \ldots, x_n)})(Q)^{-1}$.
    
     Since $Q \in V_{\ell,n}^{\reg}(h_1,\dots,h_n)$, there is a coordinate $x_i$ for which the inequality $\det\left(\frac{\partial (h_1, \ldots, h_{n-1})}{\partial (x_1, \ldots,x_{i-1},x_{i+1},\dots, x_n)}\right)(Q) \neq 0$ holds. Let $h_{n+2}$ be the $(\ell,n+2)$-$\texp$-polynomial 
    \[ x_{n+2} \det\left(\frac{\partial (h_1, \ldots, h_{n-1})}{\partial (x_1, \ldots, x_{i-1},x_{i+1},\dots x_n)}\right)(x_1,\dots,x_n)-1\]
    and let $q_{n+2}=\det\left(\frac{\partial (h_1, \ldots, h_{n-1})}{\partial (x_1, \ldots,x_{i-1},x_{i+1},\dots, x_n)}\right)(Q)^{-1}$.
    
    Then $(Q,q_{n+1},q_{n+2}) \in V_{\ell,n+2}^{\reg}(h_1,\dots,h_{n+2})$. Now for each $i=1,\dots,n+2$, let $g_i$ be the $(\ell-1,n+3)$-$\texp$-polynomial obtained by replacing each occurrence of $\texp(x_\ell)$ in $h_i$ by a new variable $x_{n+3}$, let $g_{n+3}:=x_{n+3}-\texp(x_{\ell})$, and let $q_{n+3}=\texp(q_{\ell})$. It is now easy to check that $Q'=(Q,q_{n+1},q_{n+2},q_{n+3})$ and $g_1,\dots,g_{n+3}$ satisfy the requirements.
\end{proof}

\begin{cor}\label{cor:al-for-every-l}
For all $\ell\in \N$, $A_\ell$ holds.     
\end{cor}
\begin{proof}
    $A_0$ holds by model completeness of the theory of real closed fields. The inductive step $A_{<\ell} \implies A_{\ell}$ is Proposition \ref{prop:a<l-al}.
\end{proof}

\begin{thm}\label{thm:model-complete}
    The theory $\Tres$ of definably complete restricted exponential fields is model complete. 
\end{thm}
\begin{pf}
 It suffices to show that whenever $M\subset N$ are models of $\Tres$, every quantifier-free formula over $M$ which is satisfiable in $N$ is satisfiable in $M$.   \deleted{It is enough to consider quantifier-free formulas without nested occurrences of $\texp$.} Any such formula \replaced{has }{is a disjunction of formulas of} complexity $\leq \ell$ for some $\ell$ by Lemma \ref{rem:complexity}, so we conclude by Proposition \ref{prop:al-implies-complexity} and Corollary \ref{cor:al-for-every-l}. 
\end{pf}

We conclude this section with the following corollary relating the pregeometry $\ecl$ with definable closure.

\begin{cor}\label{cor:definable-closure}
   Let $N\models \Tres$ and $C$ a subset of $N$. Then the definable closure of $C$ in $N$ coincides with $\ecl^N(C)$. 
\end{cor}
\begin{proof}
Let $a\in N$ be in the definable closure of $C$. Then there are a tuple $\bar c \subset C$ and a $\emptyset$-definable function $f$ with $a = f(\bar c)$. By model completeness, $\{a\}$ is existentially definable with parameter $\bar c$, so it is a projection of an $\texp$-variety $V$ defined over $\bar c$. By Proposition \ref{prop:servi}, $V$ contains the projection of a Khovanskii point $P$ over $\Q(\bar c)$, so $a$ must be a component of $P$. This shows that the definable closure of $C$ is contained in $\ecl^N(C)$ and the other inclusion is easy. 
\end{proof}

\section{Lifting Khovanskii points from the residue field}\label{sec:lifting}
In the rest of the paper we prove that, under a suitable hypothesis, the residue field of a convex subring of a model $N$ of $\Tres$ can be embedded in $N$. This will be used to prove the completeness of $\Tres$ and $\Tglob$ assuming $\SC$. 

Fix $N\models \Tres$ and a convex subring $A\subset N$.  

\begin{lem}\label{lem:newton}
Let $F \subseteq A$ be a subfield of $N$ and let $\underline F = \res(F) \subset \rf$ be its image under the residue map.

Let $(a,b) \in (\rf)^n$ be an $(\ell,n)$-Khovanskii point over $\underline F$ satisfying the $(\ell,n)$-Khovanskii system $(f_1, \ldots, f_n)$ in $\rf$. Then there is a unique $(\alpha,\beta) \in A^n$ with $\res(\alpha, \beta)=(a,b)$ that satisfies the corresponding system over $F$ (obtained by identifying $F$ with $\rs F$ via the residue map). 
\end{lem}

\begin{proof} Since $F$ is contained in $A$, the restriction of the residue map to $F$ is injective, so we can identify $F$ with $\underline F$ and consider $F$ as a common subfield of $N$ and $\rf$. Hence we may consider $f_1, \ldots, f_n$ as $\texp$-polynomials over $F$ and look for solutions in $N$. 

Let $(\alpha', \beta') \in A^n$ be an arbitrary point with $\res(\alpha',\beta')=(a,b)$. Points (1)--(3) below are all consequences of the fact that taking residues commutes with the ring operations and $\texp$ (Corollary \ref{cor:induced-texp}). 
   \begin{enumerate}
      \item $f_i(\alpha',\beta') \in \mathfrak{m}_A$ for $i=1,\dots,n$
      \item \label{point:Jac} For each $i,j$, $\res \left(\frac{\partial f_i}{\partial x_j}  (\alpha', \beta')\right)=\frac{\partial f_i}{\partial x_j} (a,b)$, hence the Jacobian determinant of $(f_1, \dots, f_n)$ computed at $(\alpha',\beta')$ does not lie in $\mathfrak{m}_A$.       
      \item For each $i,j,k$, $\frac {\partial f_i}{\partial x_j \partial x_k}(\alpha', \beta')\in A$ (since $(\alpha',\beta') \in A^n$ and the coefficients of the $\texp$-polynomials $\frac {\partial f_i}{\partial x_j \partial x_k}$ lie in $F$). 
\end{enumerate}
Intuitively this says that $(f_1, \ldots, f_n)(\alpha',\beta')$ is very small, the Jacobian determinant is not too small, and the second derivatives are not too large. It is easy to see that (1)--(3) ensure that the hypotheses of Newton's method for definably complete expansions of a field (see \cite[Theorem 7]{Servi2008}) are satisfied, so there is a zero $(\alpha, \beta)$ of $f_1,\dots,f_n$ with $\res(\alpha, \beta)=\res(\alpha',\beta') = (a,b)$. Moreover $(\alpha,\beta)$ must be a regular zero since its residue $(a,b)$ is a regular zero. 

For uniqueness, let $\gamma_1=(\alpha,\beta)$ and assume $\gamma_2\neq \gamma_1$ is another regular zero of $(f_1,\dots,f_n)$ with the same residue $c=(a,b)$. Let $h:=\gamma_2-\gamma_1$, and for $i=1,\dots,n$ let $g_i:[0,1] \to N$ be the function $t \mapsto f_i(\gamma_1+th)$. For every $i$ we have $g_i(0)=g_i(1)=0$, so there is $t_i \in [0,1]$ such that $g_i'(t_i)=0$. Now we have \[g_i'(t)=\frac{d}{dt}(f_i(\gamma_1+th))=\nabla f_i(\gamma_1+th) \cdot h.\] Therefore, for every $i$ we have \[0=g_i'(t_i)=\nabla f_i(\gamma_1+t_i h) \cdot h \] and thus \[\nabla f_i(\gamma_1+t_i h) \cdot \frac{h}{||h||}=0 .\] Since $\frac{h}{||h||}$ is a unit vector, its residue is non-zero. Finally, observe that \begin{align*}
    \res (\nabla f_i(\gamma_1+th))&=\nabla f_i(c)
\end{align*} 
so that \[0=\res(0)=\res\left(\nabla f_i(\alpha_1+th) \cdot \frac{h}{||h||} \right)=\nabla f_i(c) \cdot \res\left(\frac{h}{||h||}\right).\] It follows that $\res(\frac{h}{||h||}) \in \ker \Jac(f_1,\dots,f_n)(c)$, contradicting the assumption that $c$ is a regular zero of $(f_1,\dots,f_n)$.
\end{proof}
  
The next lemma is well known, see for instance \cite[Exercise 13.6]{Eisenbud1995}, so we omit the proof. The same result was used in the proof of Theorem 1.1 of \cite{Macintyre1996a}.   

\begin{lem}\label{lem:quasi-poly-division}
    Let $F$ be a field. 
		Suppose that $\mathfrak p$ is a prime ideal of the polynomial ring $F[x_1, \ldots, x_m]$ and that the field of fractions of $F[x_1, ..., x_m]/\mathfrak p$ has transcendence degree $\ell$ over $F$. Then there is $h\in  F[x_1, \ldots, x_m]$ with $h\nin \mathfrak p$ and $m-\ell$ polynomials $h_1, \ldots, h_{m-\ell}$ in $\mathfrak{p}$, such that $h\mathfrak p \subseteq (h_1, \ldots, h_{m-\ell})$. 
\end{lem}
   
The next lemma is similar to Lemma \ref{lem:newton}, with the difference that we show, under a transcendence assumption, how to uniquely lift a Khovanskii point, without specifying the system.

\begin{lem} \label{lem:unique} 
Let $F \subseteq A$ be a subfield of $N$. Let $(a,b) \in (\rf)^n$ be an $(\ell,n)$-Khovanskii point over $\rs F = \res(F)$. Suppose that $\td(a,b, \texp(a)/\rs F) = \ell$. Then there is 
 a unique $(\alpha,\beta) \in A^n$ with $\res(\alpha, \beta)=(a,b)$ such that $F(\alpha, \beta, \texp(\alpha)) \subset A$.  
\end{lem}

\begin{proof} As in the proof of Lemma \ref{lem:newton} we identify $F$ and $\underline F$. 
    Let $\mathfrak p$ be the ideal of all polynomials in $F [x_1, \ldots, x_n, y_1, \ldots, y_\ell]$ that vanish at $(a,b, \texp(a))\in (\rf)^{n+\ell}$. Since $\td(a,b,\texp(a)/ F) =  \ell$, the fraction field of $ F [\bar x, \bar y]/\mathfrak p$ has transcendence degree $\ell$ over $F$, hence by Lemma \ref{lem:quasi-poly-division}, there are $h_1, \ldots, h_n\in \mathfrak p$ and $h\in F[\bar x, \bar y]$ not in $\mathfrak p$, such that $h\mathfrak p \subseteq (h_1, \ldots, h_n)$. 
    
    Since $(a,b)$ is an $(\ell,n)$-Khovanskii point over $F$, there are $p_1, \ldots, p_n \in \mathfrak p$ such that $(a,b)$ is a regular zero of $F_{p_1}, \ldots, F_{p_n}$. We claim that $(a,b)$ is a regular zero also of the system $F_{h_1}, \ldots, F_{h_n}$.
   
    By construction, $hp_1, \ldots, hp_n$ belong to $(h_1, \ldots, h_n)$. We can then write polynomial equations of the form $hp_i = \sum_{j=1}^n q_j h_j$. Substituting $\texp(x_i)$ for $y_i$ for $i=1, \ldots, \ell$ and differentiating, we obtain that $\nabla F_{p_1} (a,b), \ldots, \nabla F_{p_n} (a, b)$ lie in the $F$-linear span of $\{\nabla F_{h_1}(a,b),\dots, \nabla F_{h_n}(a,b) \}$. Since $(a,b)$ is a regular zero of $(F_{p_1}, \ldots, F_{p_n})$, the vectors $\nabla F_{p_i}(a,b)$ are $F$-linearly independent, and therefore the vectors $\nabla F_{h_i}(a,b)$ must also be $F$-linearly independent, namely $(a,b)$ is a regular zero of $F_{h_1}, \ldots, F_{h_n}$, as desired. 
    
    By Lemma \ref{lem:newton}, we can lift $(a,b) \in (\rf)^n$ to a unique (necessarily regular) zero $(\alpha,\beta) \in A^n$ of the system $F_{{h_1}},\ldots, F_{{h_n}}$, so in particular $(\alpha, \beta, \texp(\alpha))$ is a zero of $(h_1, \ldots, h_n)$. Let us also observe that $h(\alpha,\beta,\texp(\alpha)) \neq 0$ because its residue $h(a,b,\texp(a))$ is not zero (as $h\nin \mathfrak p$). 
    
    We claim that there is section
    $s: F[a,b,\texp(a)] \to F[\alpha,\beta,\texp(\alpha)]$ of the residue map that is a ring homomorphism. To this end, we must show that if $p$ is a polynomial over $F$ that vanishes at $(a,b,\texp(a))$, then $p$ also vanishes at $(\alpha, \beta, \texp(\alpha))$ when evaluated in $N$. Assuming the antecedent, we have $hp \in (h_1, \ldots, h_n)$. Since $h_i(\alpha, \beta,\texp(\alpha)) = 0$ for all $i$ and $h(\alpha, \beta, \texp(\alpha)) \neq 0$, it follows that $p(\alpha, \beta, \texp(\alpha)) = 0$, as desired. The existence of the section $s$ implies that 
    $F(\alpha,\beta,\texp(\alpha)) \subset A$.
    
    \deleted{It remains to observe that this inclusion implies that $(\alpha,\beta)$ is the unique lifting to $A^n$ of any $(\ell,n)$-Khovanskii system realized by $(a,b)$ in $\rf$, so $(\alpha,\beta)$ is uniquely determined.} \added{For the uniqueness, suppose $F(\alpha',\beta',\texp(\alpha')) \subset A$ and $\res(\alpha',\beta')=(a,b)$. It follows that, for every $i=1,\dots,n$, we have $\res(F_{h_i}(\alpha',\beta'))=F_{h_i}(a,b)=0$ and since $F(\alpha',\beta',\texp(\alpha')) \subset A$ this implies that $F_{h_i}(\alpha',\beta')=0$. Since the Jacobian of the system $F_{h_1}=\dots=F_{h_n}=0$ computed at $(a,b)$ is non-zero, it is also non-zero when computed at $(\alpha',\beta')$. Hence $(\alpha',\beta')$ is a solution of the same Khovanskii system as $(\alpha,\beta)$ with the same residue, so $(\alpha',\beta')=(\alpha,\beta)$ by Lemma \ref{lem:newton}.} 
\end{proof}

\begin{conj}[Weak Schanuel's Conjecture \cite{Macintyre1996a}, Section 5]
\added{    There is a computable function $q$ that, given a Khovanskii system $H$ over $\Q$ and an $\texp$-polynomial $g$ over $\Q$, both in $n$ variables, yields a rational number $q = q(n,H,g) >0$, such that for any solution $P \in \R^n$ of $H$, either $g(P) = 0$, or $|g(P)|>q$.} 
\end{conj}

\added{We abbreviate Weak Schanuel's Conjecture by WSC. Note that we state it for the restricted exponential function while in \cite{Macintyre1996a} it is stated for the total exponential, but the statements are easily seen to be equivalent since we are working in the reals.}         

\begin{defn}\label{def:Twsc}
    \added{Assume WSC and let $q$ be a computable function as given by the conjecture. Let $\Twsc$ denote an axiom scheme which says that for every $H,g$ and $n$ as in the conjecture and every solution $P$ of $H$, either $g(P)=0$ or $|g(P)|>q(n,H,g)$.}
\end{defn}

\added{Assuming WSC, we can relax the transcendence assumption in Lemma \ref{lem:unique} as follows.}

\begin{lem}\label{lem:uniquewsc}
    \added{Assume WSC. Let $A = \Fin(N)$, and let $(a,b)\in (\rf)^n$ be an $(\ell, n)$-Khovanskii point over $\Q$. If $N\models \Twsc$, then there is a unique $(\alpha,\beta)\in A^n$ with $\res(\alpha,\beta) = (a,b)$ such that $\Q(\alpha, \beta, \texp(\alpha)) \subset A$.} 
\end{lem}
\begin{proof}
\added{By Lemma \ref{lem:newton} there is a Khovanskii point $(\alpha,\beta)\in  N$ over $\Q$ whose residue is $(a,b)$. If $\Q(\alpha,\beta,\texp(\alpha)) \not\subseteq A$ there is a polynomial $g$ over $\Q$ such that $g(\alpha, \beta,\texp(\alpha))$ is a non-zero infinitesimal. This contradicts the fact that $N \models \Twsc$. For uniqueness we argue as in the last paragraph of the proof of Lemma \ref{lem:unique}.}
\end{proof}

\section{Completeness}\label{sec:completeness}

Point (1) of the following Lemma will be used in the proof of completeness of $\Tres$. Point (2) will be needed in the next section. 

Given a tuple $t=(t_1,\dots,t_k)$, we write $\texp(t)^{\mathbb{Q}}$ for the infinite set $\{\texp(t_i)^{q}|i=1,\dots,k, \, q \in \Q\}$.

\begin{lem}\label{lem:eliminate-dependences}
Let $N$ be a restricted exponential field, not necessarily definably complete. 
\begin{enumerate}
    \item Let $(a,b)\in N^n$ be an $(\ell,n)$-Khovanskii point over $\Q$. Suppose that $a_\ell$ is in the $\Q$-linear span of $a_1,\dots,a_{\ell-1}$. Then there is a non-zero $d \in \Z$ such that, letting $a'=(a_1/d, \ldots, a_{\ell -1}/d)$,  we have that $(a',b)$ is an $(\ell-1, n-1)$-Khovanskii point over $\Q$. \label{item:el-dep-1}
    \item Let $t = (t_1, \ldots, t_k)$ be a finite tuple from $N$ (possibly empty), $M$ a restricted exponential subfield of $N$, and $(a,b) \in N^n$ an $(\ell,n)$-Khovanskii point over $M(t,\texp(t)^\Q)$. Suppose that $a_\ell$ is contained in the $\Q$-linear span of $a_1, \ldots, a_{\ell-1}$ modulo $\spn_\Q(M \cup t)$. Then there is a non-zero $d\in \Z$ such that, letting $a' = (a_1/d, \ldots, a_{\ell -1}/d)$,  we have that $(a',b)$ is an $(\ell-1, n-1)$-Khovanskii point over $M(t,\texp(t)^{\Q})$. \label{item:el-dep-2}    \end{enumerate}
\end{lem}

\begin{pf}
Since (2) is more complicated, we prove (\ref{item:el-dep-2}) and we indicate how to modify the proof to obtain (1). 

By the assumptions we can write 
$$da_\ell = \sum_{j\leq k} r_j t_j +  \sum_{i<\ell} k_i a_i + g$$
for some $r_j, d,k_i\in \Z$ with $d\neq 0$ and some $g\in M$.  Since $t_1, \ldots, t_k, a_1, \ldots, a_\ell \in (-1,1)$, there is $m\in \N$ such that $g/dm\in (-1,1)$, so we can write $e^\frac{g}{d}$ for the element $\texp(g/dm)^m \in M$.  
By the assumptions $(a,b)$ is a regular zero of a system of $\texp$-polynomials $F_{p_1}(\bar x, \bar u), \ldots, F_{p_n}(\bar x, \bar u)$ over $M(t,\texp(t)^{\Q})$ in the variables $\bar x = (x_1, \ldots, x_\ell)$ and $\bar u = (u_1, \ldots, u_{n-\ell})$.  Let $\bar X = (X_1, \ldots, X_{\ell-1})$ be a new list of variables and consider the substitution
$x_i = dX_i$ for $i<\ell$ and 
$$x_\ell = \sum_{j\leq k} r_j \tfrac {t_j} d + \sum_{i<\ell} k_iX_i + \tfrac g d.$$
The differential of the map 
$$L(\bar X, \bar u) = (\bar x, \bar u)$$
has rank $n-1$ and $L(a',b)=(a,b)$, so the differential of $(F_{p_1}, \ldots, F_{p_n}) \circ L$ has rank $n-1$ at $(a',b)$. 
 
Under the substitution $L$,  
$$\texp(x_\ell) = e^{\frac gd} \prod_{j \leq \ell} \texp(t_j)^{r_j/d}\prod_{i<\ell}\texp(X_i)^{k_i}$$
and $\texp(x_i) = \texp(X_i)^{d}$ for $i<\ell$. Moreover, $e^{\frac{g}{d}} \in M$ (because $M$ is closed under $\texp$), so choosing $n_0\in \N$ sufficiently large one can find $\texp$-polynomials $f_i(\bar X, \bar u)$ over $M 
(t, \texp(t)^\Q)$ such that 
$$(F_{p_i} \circ L) (\bar X, \bar u) = \frac{f_i(\bar X, \bar u)}{\texp(t)^{n_0} \texp(X)^{n_0}}.$$ 
Since $\texp(t)$ and $\texp(X)$ are never zero, the rank of the differential of $(f_1, \ldots, f_n)$ at $(a',b)$ coincides with the rank of $(F_{p_1}, \ldots, F_{p_n})\circ L$ at $(a',b)$ and it is therefore equal to $n-1$. By omitting a suitable $f_i$ we obtain a non-zero minor of rank $n-1$, hence a system of $\texp$-polynomial witnessing the fact that $(a',b)$ is a $(\ell-1,n-1)$-Khovanskii point over $M 
(t, \texp(t)^\Q)$. This completes the proof of (2).

The proof of (1) follows the same strategy, but it is simpler: we write $da_\ell=\sum_{i <\ell} k_ia_i$ and obtain $\texp$-polynomial $F_{p_1},\dots,F_{p_n}$ over $\Q$. The rest of the modifications is clear.
\end{pf}

\begin{thm}\label{thm:complete-restricted}
Assume $\SC$. Then $\Tres$ is complete.  
\end{thm}

\begin{pf} Since $\Tres$ is model complete (Theorem \ref{thm:model-complete}), it suffices to show that every $\aleph_0$-saturated model $N$ of $\Tres$ contains a substructure $M \subset N$ that is elementarily equivalent to $\R_{\texp}$. So let $N\models \Tres$ be $\aleph_0$-saturated. Let  $\Fin(N) \subset N$ be the subring of finite elements and let $\mu\subset \Fin(N)$ be its maximal ideal. For every $\alpha \in A$ there is a unique real number $a\in \R$ that determines the same Dedekind cut over $\Q$ and we say that $a := \st(\alpha)$ is the {\em standard part} of $\alpha$. By saturation $\st:\Fin(N) \to \R$ is surjective. 
By Proposition \ref{prop:archimedean-residue}, we can identify $\Fin(N)/\mu$ with $\R_{\texp}$ and the standard part map $\st$ with the residue map $\res:\Fin(N) \to \Fin(N)/\mu$. Let 
$\ecl^{\R_{\texp}}(\emptyset)$ be the $\ecl$-closure of $\emptyset$ in $\R_{\texp}$. In $\R_{\texp}$ the exponential closure coincides with the definable closure (see \cite[Theorem 4.2]{Jones2008} or Corollary \ref{cor:definable-closure}), so $\ecl^{\R_{\texp}}(\emptyset)$ is an elementary substructure of $\R_{\texp}$. Every element of $\ecl^{\R_{\texp}}(\emptyset)$ is a component of an $(\ell,n)$-Khovanskii point $(a,b)\in \R^n$ over $\Q$. Fix such a point $(a,b) \in \R^n$ and suppose that 
the components of $a \in \R^\ell$ are $\Q$-linearly independent. By $\SC$ \added{and Lemma \ref{lem:exists-variety}}, we have $\td(a,b,\texp(a)/\Q) = \ell$. By Lemma \ref{lem:unique} it then follows that there is a unique tuple $(\alpha,\beta) \in A^n$ with (componentwise) residue $(a,b) = \st(\alpha, \beta) \in \R^n$ such that $\Q(\alpha,\beta,\texp(\alpha))\subset A$. The same holds without the assumption that the components of $a$ are $\Q$-linearly independent because we can reduce to that case by a repeated application of Lemma \ref{lem:eliminate-dependences}(\ref{item:el-dep-1}). 
Letting $a,b$ vary, \added{the fields $\Q(\alpha, \beta, \texp(\alpha))$ form an upper directed family because the concatenation of two Khovanskii points is again a Khovanskii point.} The union $P$ of the \added{family} \deleted{fields $\Q(\alpha, \beta, \texp(\alpha))$} is \added{a substructure of $N$} contained in $\Fin(N)$, and it is isomorphic to $\ecl^{\R_{\exp}}(\emptyset)$ under the residue map, so in particular it is elementarily equivalent to $\R_{\texp}$.  
\end{pf}

\begin{cor}\label{cor:prime-model-embeds}
    \added{Assume $\SC$. Then $\ecl^{\Rrexp}(\emptyset)$ embeds elementarily in every model of $\Tres$.} 
\end{cor}
\begin{proof}
    \added{Let $M\models \Tres$. It suffices to show that $\ecl^M(\emptyset)$ is isomorphic to $\ecl^{\Rrexp}(\emptyset)$. This follows from the completeness of $\Tres$ and the fact that the $\ecl$-closure coincides with the definable closure (Corollary \ref{cor:definable-closure}).} 
\end{proof}

\begin{thm} \label{thm:complete-unrestricted}
 Assume $\SC$. Then   $\Tglob$ is complete. 
\end{thm}
\begin{proof}
    \replaced{By}{This follows from} the completeness of $\Tres$ and Ressayre's theorem (Theorem \ref{prop:ressayre}). 
\end{proof}

\begin{rem}\mbox{} \label{rem:Krapp}
\begin{enumerate}
\item \added{The proof of Theorem \ref{thm:complete-restricted} can be modified, replacing the application of Lemma \ref{lem:unique} with one of Lemma \ref{lem:uniquewsc}, to show that WSC implies that $\Tres+\Twsc$ is complete. As in Theorem \ref{thm:complete-restricted} completeness can also be obtained for $\Tglob+\Twsc$.} \label{rem:WSC}
\item In an attempt to prove Theorem \ref{thm:complete-unrestricted}, Krapp proved, assuming $\SC$, that the prime  model of $T_{\exp}$ can be embedded into any  o-minimal exponential field \cite{Krapp2023}, but the embedding could not be proved to be elementary. 
\item The saturation hypothesis in the proof of Theorem \ref{thm:complete-restricted} was introduced to ensure that $\Fin(M)/\mu$ is a model of $T_{\exp}$, but this also follows without saturation by \cite[Theorem 3.3]{Krapp2019}. For the sake of the theorem, the saturation assumption can be assumed without loss of generality. 
\end{enumerate}

\end{rem}

\section{Embedding the residue restricted exponential field}\label{sec:embedding}

In this section we prove an embedding result for the residue field of a convex subring of a model of $\Tres$, similar to the one implicitly used in the proof of Theorem \ref{thm:complete-restricted}. The theorem is reminiscent of the results in \cite{DriesLewenberg1995}, but the assumptions are weaker. In particular, we do not assume that the convex subring $A\subset N$ is $T$-convex in the sense of that paper.

\begin{thm}\label{thm:embedding}
\replaced{Let $N \models \Tres$ and let $A \subseteq N$ be a convex subring with maximal ideal $\mathfrak{m}_A$. Assume that $A$ contains an elementary substructure $C$ of $N$. Then there is a restricted exponential subfield $R_A$ of $N$ containing $C$ such that $A=R_A+\mathfrak{m}_A$. }{
Let $N \models \Tres$ and let $A \subseteq N$ be a convex subring with maximal ideal $\mathfrak{m}_A$. Assume that $A$ contains a substructure $C$ of $N$ whose image under the residue map $\res:A\to A/\mathfrak{m}_A$ is $\ecl$-closed in $\rf$. Then there is a restricted exponential subfield $R_A$ of $N$ containing $C$ such that $A=R_A+\mathfrak{m}_A$. 
    
    In particular, if $\SC$ holds or if $A$ contains an elementary substructure of $N$, the residue field $A/\mathfrak{m}_A$ can be embedded in $N$ through a section of the residue map that preserves $\texp$. 
}
\end{thm}

\begin{pf} 
\deleted{First notice that if $\SC$ holds, then one can take for $C$ a copy of the prime model of $T_{\exp}$ and embed it in $N$ as in the proof of Theorem \ref{thm:complete-restricted}. Moreover, it is not difficult to see that the residue of an elementary substructure contained in $A$ is $\ecl$-closed. It is then sufficient to prove the first part of the statement.}
\added{By model completeness $C$ is $\ecl$-closed in $N$ (Theorem \ref{thm:model-complete}). By Lemma \ref{lem:newton}, it follows that $\rs{C} = \res(C)$ is $\ecl$-closed in $A/\mathfrak{m}_A$.}
By Proposition \ref{prop:pregeometry} the exponential closure defines a pregeometry on $A/\mathfrak{m}_A$, so we can fix a $\ecl$-basis $\mathcal B$ of $\rf$ over $\rs C = \res(C)$. We assume without loss of generality that each element of the basis is in $(-1,1)$ (we can reduce to this case taking the reciprocal of the elements not in $(-1,1)$). Now lift $\mathcal B$ to a family $\widehat {\mathcal B}$ of elements of $A$ that is mapped bijectively to $\mathcal B$ via the residue map. For each $c \in \rf$ there is a finite tuple $t = (t_1, \ldots, t_k)$ from $\mathcal B$ such that $c \in \ecl(\rs Ct)$. This implies that there are $\ell,n\in \N$ and an $(\ell, n)$-Khovanskii point $(a,b)\in (\rf)^n$ over $\rs C (t)$ such that $c$ is a component of $(a,b)$. 

In particular this shows that $\rf$ is a union of subfields of the form 
$$\rs C_{t,a,b} = \rs C(t, a, b, \texp(t)^{\Q}, \texp(a)) \subset \rf$$
where:
\begin{itemize}
    \item $t = (t_1, \ldots, t_k)$ is a $k$-tuple from $\mathcal B$ and $t \subset (-1,1)$.
    \item $(a,b)$ is an $(\ell,n)$-Khovanskii point over $\rs C (t,\texp(t)^{\Q})$.
    \item $\texp(a) = (\texp(a_1), \ldots, \texp(a_\ell))$ where $a_i$ is the $i$-th component of $a$. 
    \item $\texp(t)^{\Q}$ is the set of all elements of $\rf$ of the form $\texp(t_i)^{q}$ where $t_i\in \mathcal B$ is a component of $t$ and $q\in \Q$.
    \item $k,\ell,n \in \N$. 
\end{itemize}

Now fix $t,a,b$ as above and let
$\tau = (\tau_1, \ldots, \tau_k)$ be the unique $k$-tuple in $\widehat {\mathcal B}$ with $\res(\tau) = t$. 
We will prove that there is a unique $n$-tuple $(\alpha, \beta) \in A^n$ with $\res(\alpha,\beta) = (a,b)$ such that  
$$C_{\tau,\alpha,\beta} = C(\tau, \alpha, \beta, \texp(\tau)^{\Q}, \texp(\alpha)) \subset A$$
where the operations are performed in $N$. 

\begin{claim} \label{claim:td2k}
  $\td(t,\texp(t)/\rs C) = 2k$.   
\end{claim}
\begin{pf}
Since $\rs C$ is $\ecl$-closed, by the predimension inequality (Theorem \ref{thm:kirby}), we have $\td(t,\texp(t)/\rs C)\geq \ldim_{\mathbb{Q}}(t/\rs C)+\etd(t/\rs C)=2k$. The other inequality is clear since $(t,\texp(t))$ is a tuple of length $2k$. 
\end{pf}

\begin{claim} \label{claim:lift2k}
$C(\tau, \texp(\tau)) \subseteq A$.
\end{claim}
\begin{pf}
By Claim \ref{claim:td2k}, the components of $(t, \texp(t))$ are algebraically independent over $\rs C$. Thus there is a ring homomorphism $s:\rs C [t,\texp(t)] \to C[\tau, \texp(\tau)]$ which is a partial section of the residue map, hence it is an isomorphism. 
In particular $C[\tau,\texp(\tau)]$ contains no elements of $\mathfrak{m}_A$, so its fraction field is contained in $A$.   
\end{pf}

\begin{claim}\label{claim:generic-Khovanskii}
There is a unique tuple $(\alpha,\beta) \in A^n$ with residue $(a,b) \in (\rf)^n$ such that $C(\alpha,\tau,\beta,\texp(\alpha), \texp(\tau))\subset A$.  
\end{claim}
\begin{proof} 

By a repeated application of Lemma \ref{lem:eliminate-dependences}(\ref{item:el-dep-2}), we can assume that the components of $a \in (\rf)^\ell$ are $\Q$-linearly independent modulo $\spn_\Q(\rs C \cup t)$. (This also includes the case when $\ell = 0$ and $a$ is the empty tuple).
 
Granted the linear independence, by the predimension inequality (Theorem \ref{thm:kirby}), 
\begin{align*}
    \td(a,t,\texp(a),\texp( t) /\rs C) & \geq \ldim_{\mathbb{Q}}(a,t/\rs C)+\etd(a,t/\rs C) \\
    & = (\ell+k) + k \\ & = \ell + 2k.
\end{align*}
and since $\td(t, \texp(t)/\rs C) = 2k$ we obtain
$\td(a,\texp(a)/\rs C(t, \texp(t)) \geq \ell$. By Lemma \ref{lem:exists-variety} the other inequality also holds because $(a,b)$ is an $(\ell,n)$-Khovanskii point over $\rs C(t,\texp(t)^\Q)$, so $\td(a,\texp(a)/\rs C (t, \texp(t)^{\Q}) =  \ell$. By Lemma \ref{lem:unique}, applied to $F = C(\tau, \texp(\tau)^\Q)$, it follows that $C(\alpha,\tau,\beta,\texp(\alpha), \texp(\tau))\subset A$.  
\end{proof} 

To finish the proof, note that the fields of the form $C_{\tau, \alpha, \beta}$ form an upper directed system with
$$C_{\tau,\alpha,\beta} \cup C_{\tau',\alpha',\beta'} \subseteq C_{\tau \cup \tau',\alpha\alpha',\beta\beta'}.$$
We can thus define $R_A$ as the union of these fields and note that, by construction, $R_A$ is closed under $\texp$. 
\end{pf}

\begin{cor}
    \added{Assume $\SC$. Let $N \models \Tres$ and let $A \subseteq N$ be a convex subring with maximal ideal $\mathfrak{m}_A$. Then there is a restricted exponential subfield $R_A$ of $N$ such that $A=R_A+\mathfrak{m}_A$.}
\end{cor}
\begin{proof}
\added{By Corollary \ref{cor:prime-model-embeds}, if $\SC$ holds then $N$ contains an archimedean elementary substructure, which is then contained in every convex subring of $N$. It follows that the assumptions of Theorem \ref{thm:embedding} are satisfied.}
\end{proof}

\begin{rem}
\added{If $\SC$ holds, then every model of $\Tres$ contains a copy of the prime model of $T_{\exp}$ as shown in the proof of Theorem \ref{thm:complete-restricted}}
\end{rem}

\bibliographystyle{alpha}

\listofchanges

\end{document}